\documentclass[12pt]{article}
\usepackage{amsmath}
\usepackage{latexsym}
\usepackage{amssymb}
\usepackage{a4wide}
\newtheorem{thm}{Theorem}[section]
\newtheorem{la}[thm]{Lemma}
\newtheorem{Defn}[thm]{Definition}
\newtheorem{Remark}[thm]{Remark}
\newtheorem{Note}[thm]{Note}
\newtheorem{prop}[thm]{Proposition}

\newtheorem{cor}[thm]{Corollary}
\newtheorem{Example}[thm]{Example}
\newtheorem{Examples}[thm]{Examples}
\newtheorem{Problems}[thm]{Problems}

\newtheorem{Problem}[thm]{Problem}
\newtheorem{Number}[thm]{\!\!}
\newenvironment{defn}{\begin{Defn}\rm}{\end{Defn}}

\newenvironment{rem}{\begin{Remark}\rm}{\end{Remark}}
\newenvironment{numba}{\begin{Number}\rm}{\end{Number}}
\newenvironment{proof}{{\noindent\bf Proof.}}%
                  {\nopagebreak\hspace*{\fill}$\Box$\medskip\medskip\par}   
\newcommand{\Punkt}{\nopagebreak\hspace*{\fill}$\Box$}

\newcommand{\wb}{\overline}
\newcommand{\ve}{\varepsilon}

\newcommand{\wt}{\widetilde}

\newcommand{\n}{\rm}

\newcommand{\mto}{\mapsto}

\newcommand{\isom}{\cong}

\newcommand{\N}{{\mathbb N}}
\newcommand{\R}{{\mathbb R}}
\newcommand{\bL}{{\mathbb L}}

\newcommand{\Z}{{\mathbb Z}}
\newcommand{\C}{{\mathbb C}}
\newcommand{\K}{{\mathbb K}}

\newcommand{\cU}{{\cal U}}
\newcommand{\cO}{{\cal O}}
\newcommand{\cV}{{\cal V}}

\newcommand{\dl}{{\displaystyle \lim_{\longrightarrow}}}

\newcommand{\Hom}{\mbox{\n Hom}}

\newcommand{\take}{\backslash}
\newcommand{\sub}{\subseteq}

\newcommand{\GL}{\mbox{\rm GL}}

\newcommand{\im}{\mbox{\n im}}

\newcommand{\pr}{\mbox{\rm pr}}

\newcommand{\id}{\mbox{\n id}}

\newcommand{\cB}{{\cal B}}
\newcommand{\cA}{{\cal A}}
\newcommand{\cD}{{\cal D}}

\newcommand{\cK}{{\cal K}}

\newcommand{\cP}{{\cal P}}

\newcommand{\smim}{\mbox{\n\footnotesize im}}
\newcommand{\sSup}{\mbox{\n\footnotesize supp}}

\newcommand{\Int}{\mbox{\n int}}

\newcommand{\Supp}{\mbox{\n supp}}

\newcommand{\sbull}{{\scriptscriptstyle \bullet}}
\begin{document}
\begin{center}
{{\large\bf
Differentiable Mappings Between Spaces of Sections
}}\vspace{2mm}\\
{\bf Helge Gl\"{o}ckner}
\end{center}
\begin{abstract}
\hspace*{-6 mm}We study differentiability properties of
mappings between
spaces of sections
in\linebreak
locally trivial bundles of locally convex spaces,
and describe various applications.\linebreak
In particular, we show that the space of $C^r$-sections
$C^r(X,E)$ is a topological $C^r(X)$-module,
for every
$C^r$-vector bundle $E\to X$ (where $r\in \N_0\cup\{\infty\}$),
and we show that,
provided $X$ is finite-dimensonal
and $\sigma$-compact,
the space $\cD^r(X,E)$ of compactly supported $C^r$-sections
is a topological module over the topological algebra
$\cD^r(X)$ of test functions of class~$C^r$ on~$X$.
We also show that the mapping group
$C^r(X,G)$ can be
given a smooth (resp., $\K$-analytic)
Lie group structure modelled
on $\cD^r(X,L(G))$,
for every (possibly infinite-dimensional)
smooth (resp., $\K$-analytic)
Lie group~$G$, where $\K\in\{\R,\C\}$.
\end{abstract}
\noindent
{\footnotesize
{\bf Classification:\/}
22E67, 46E25, 46T20, 58C25.\\
{\bf Keywords and Phrases:\/} vector bundle, section,
differentiable map, pushforward, bundle map,
universal central extension, mapping group, test function group,
current group, topological module.}\vspace{-2.5 mm}\\

%
%
%
%
%
%
%
%
% CLASSIFICATION:
%
% 22E---Lie groups:
%
% 22E65: Infdim Lie gps and their Lie algs
% -->22E67: Loop gps and related constructions, gp-theoretic treatment
%
% 46E: Linear function spaces and their duals:
%
% 46E40: Spaces of vector/operator-valued functions
% 46E50: Spaces of diff`ble or holom fctns on inf-dim spaces
% --> 46E25: Rings and algs of cts, diff'ble or analytic fctns
%
% 46F: Distr, gen fctns, distr spaces:
%
% 46F05: Top lin spaces of test fctns, distr and ultradistr
%
% 46T: non-linear FA:
%
% 46T05: inf-dim mfds
% 46T10: mfds of mappings
% --> 46T20: cts and diff`ble maps
%
%
% 58: global analysis, analysis on manifolds
%
% 58C: Calculus on manifolds, non-linear operators
%
% --> 58C25: diff`ble maps
%
%
%
%
%
\begin{center}
{\large\bf Introduction}\vspace{1 mm}
\end{center}
In this article, we study differentiability
properties of mappings between spaces of sections,
motivated by applications in the theory
of infinite-dimensional Lie groups.
For example,
we are interested in
differentiability properties
of the pushforward
\[ C^\infty(X,f)\!:\; C^\infty(X,E_1)\to C^\infty(X,E_2),
\;\;\;\; \sigma\mto f\circ \sigma\]
of smooth sections
in
smooth locally trivial bundles~$E_1$ and~$E_2$
of locally convex spaces
over the same $\sigma$-compact,
finite-dimensional base~$X$
(and the corresponding pushforward
$\cD^\infty(X,f)\!: \cD^\infty(X,E_1)\to \cD^\infty(X,E_2)$
of compactly supported, smooth sections),
when $f\!:E_1\to E_2$
is a fibre-preserving smooth mapping
taking the zero-section to
the zero-section
(but not necessarily linear on fibres).
As special cases of
Theorem~\ref{pushforw1} and
our main result, Theorem~\ref{pushforw3},
we obtain:\\[2mm]
{\em The mappings
$\cD^\infty(X,f)$ and $C^\infty(X,f)$ are smooth}.\\[2mm]
More generally, we consider mappings
between $C^r$-vector bundles,
and, in the case of $\cD^r(X,f)$, allow $f$ to be defined
not globally but on
an open neighbourhood
of the zero-section only.
Smoothness of mappings (or being of class~$C^r$) is understood
in the sense of Michal-Bastiani
(i.e., we are dealing with Keller's $C^r_c$-maps
in the terminology of \cite{Kel}).
As differentiability properties
of pushforwards are well-understood
when $X$ is compact,
we are primarily
interested in the case where
$X$ is
a $\sigma$-compact,
{\em non-compact\/}, finite-dimensional manifold.\vspace{2 mm}\\
Our main results
can be seen as variants
of \cite{KaM}, Corollary~30.10,
which establishes
smoothness
{\em in the sense of convenient differential calculus\/}
of pushforwards of smooth compactly supported
sections;
and also of the $\Omega$-Lemma
\cite[Theorem~8.7]{Shi} (for finite-dimensional smooth vector bundles,
in the setting of Keller's $C^\infty_c$-maps).
By contrast, our investigations
(within Keller's $C^r_c$-theory)
subsume
the study of pushforwards
of $C^r$-sections (compactly supported
or not) also for finite $r\in \N_0$,
and for possibly infinite-dimensional fibres.\\[2mm]
We have two main applications in mind.
For convenience, let us assume for the remainder of
this introduction that
$X$ is a $\sigma$-compact, finite-dimensional
smooth manifold
(although part of the results will hold
more generally),
and fix $r\in \N_0\cup\{\infty\}$.\vspace{3 mm}\\
{\bf Application~1.\/}
It readily follows from
our main results that
the function spaces
$C^r(X,A)$ and $\cD^r(X,A)$
are topological algebras
under pointwise operations,
for every
locally convex,
associative topological algebra~$A$
(see~\cite{GOO} for further investigations
of such algebras, with a view towards
applications in infinite-dimensional
Lie theory).
Thus $C^r(X)$ and $\cD^r(X)$ are
topological algebras in particular
(cf.\ already~\cite{Mic}).
We also conclude that $C^r(X,E)$
is a topological $C^r(X)$-module
and $\cD^r(X,E)$ a topological
$\cD^r(X)$-module,
for every $C^r$-vector bundle \mbox{$E\to X$.}
Thus the space
$\cD^\infty(X,T^*X)$ of compactly
supported smooth $1$-forms on~$X$ is a topological $\cD^\infty(X)$-module
in particular. This observation is useful
in infinite-dimensional Lie theory;
it is applied in~\cite{Mai} and~\cite{CUR} to prove the existence
of a universal central extension of the test function
group $\cD^\infty(X,G):=\{\gamma\in
C^\infty(X,G)\!: \gamma^{-1}(G\,\take\{1\})
\;\mbox{is relatively compact}\,\}$,
for any semi-simple finite-dimensional Lie group~$G$.\vspace{3 mm}\\
{\bf Application~2.\/}
If $G$ is a smooth or $\K$-analytic Lie group
(where $\K\in\{\R,\C\}$),
modelled on an arbitrary locally convex space,
then $\cD^r(X,G)$ carries a smooth (resp.,
$\K$-analytic) Lie group structure
modelled on $\cD^r(X,L(G))$
in a natural way (\cite{Glo};
cf.\ \cite{Alb} when $\dim(G)<\infty$).
As a second application of our main result,
we show that $C^r(X,G)$
can be given a smooth (resp., $\K$-analytic)
Lie group structure making $\cD^r(X,G)$
an open subgroup.
In the simple special case where~$G$ is a
$\K$-analytic Baker-Campbell-Hausdorff Lie group,
the $\K$-analytic Lie group structure
on $C^r(X,G)$ (which is again BCH)
had already been constructed in~\cite{Glo},
but the technical tools developed there were
too limited to pass to general Lie groups~$G$.
We remark that $C^\infty(X,G)$ had already
been given a smooth Lie group structure
in the sense of convenient differential
calculus (cf.\ \cite{KaM}, Theorem~42.21),
for every smooth Lie group~$G$ in this sense.\\[2mm]
The article is structured as follows.
Sections~1--3 are of a preparatory nature.
We describe the general setting of differential calculus,
collect what we need
concerning differentiability
of mappings between function spaces,
and prove some basic facts concerning spaces
of sections in locally trivial
bundles of locally convex spaces,
and their topologies.
The serious work begins with Section~\ref{sec4},
where the pushforward
$C^r(X,f)$ of $C^r$-sections is proved to be of class~$C^k$
under natural conditions, for globally defined~$f$
(Theorem~\ref{pushforw2}).
For~$f$ defined on an open neighbourhood of the zero-section,
analogous conclusions are established
for the pushforward $C^{\,r}_K(X,f)$
of $C^r$-sections supported in a given
compact subset~$K\sub X$
(Theorem~\ref{pushforwK}),
and finally for $\cD^r(X,f)$
in Section~\ref{sec5}
(Theorem~\ref{pushforw3}).
Stronger conditions
ensure that $\cD^r(X,f)$ is
complex analytic,
or real analytic on a zero-neighbourhood
(Section~\ref{sec6}).
The remainder of the article
is devoted to applications,
as described above. In Section~\ref{sec7},
we discuss continuity of the various natural module
structures on the various types of spaces of sections
(or distribution sections) in vector
bundles. In particular, we shall see
that $C^\infty(\R)$ and $\cD(\R)$ are topological
algebras, whereas the natural
multiplication maps $C^\infty(\R)\times \cD(\R)\to
\cD(\R)$, $C^\infty(\R)\times \cD'(\R)\to\cD'(\R)$,
and $\cD(\R)\times \cD'(\R)\to\cD'(\R)$
are hypocontinuous and thus sequentially continuous,
but discontinuous. Section~\ref{sec8}
describes an application of
the continuity of the $\cD(X)$-module
structure on spaces of compactly
supported sections.
Finally, in Section~\ref{secmapgps},
we show that $C^r(X,G)$ can be made a smooth
(resp., $\K$-analytic) Lie group modelled
on $\cD^r(X,L(G))$,
for every smooth (resp., $\K$-analytic) Lie group~$G$
(Theorem~\ref{dstruckonCr}),
and show that $C^r(X,\sbull)$ is functorial
(Proposition~\ref{Cisfuncto}).\\[2mm]
{\bf Remark} (added in 2013).
Except for updates of the references
and the introduction, this is an unpublished
manuscript dating back to 2002.
I decided to make the manuscript publicly available
as it has been used by some authors.
I should mention
that, in the meantime, I got to know
alternative techniques to deal with the problems
discussed here in a more efficient way.
I want to mention two of these:
First of all, one can easily avoid the use of cutoff-functions
and partitions of unity (as applied in this preprint)
if one uses results concerning pushforwards
like \cite[Proposition 4.23]{ZOO} (with restriction to a relatively compact subset)
instead of the results concerning pushforwards of $C^r_K$-maps
from~\cite{Glo}
(as applied in the current preprint).
More importantly,
one can exploit that the mappings $f_*:=\cD^r(X,f)\colon \cD^r(X,E_1)\supseteq \cD^r(X,U) \to \cD^r(X,E_2)$
considered
in this preprint are \emph{local} maps
in the sense that $f_*(\sigma)(x)$ only depends
on the germ of $\sigma$ at~$x$.
Therefore (as shown in \cite{LOC} and \cite{ZOO}; cf.\ also
\cite{DRN}), the map
$f_*$ will be $C^k$ if its restriction
$C^r_K(X,f)$ to
the set $C^r_K(X,E_1)\cap \cD^r(X,U) =C^r_K(X,U)$
is $C^k$ for each compact subset $K\sub X$.
This condition is much easier to prove,
and makes it unnecessary to work explicitly with
bases of the topology on spaces of compactly
supported sections (which are used in this preprint).\\[2mm]
Summing up, I recommend to check~\cite{ZOO}
for possible superior versions
of some of the results, and easier strategies of proof.
This source is, however, made more complicated to read in other ways, because
of its added generality
(as the real and complex ground fields are replaced
by more general topological fields).
A streamlined exposition of essential facts
in the real and complex cases will be made available in~\cite{GaN}.
\section{Basic definitions and facts from differential calculus}\label{sec1}
We use the framework of
differential calculus of smooth and analytic
mappings between open subsets of arbitrary locally convex spaces
as described in~\cite{RES}
(cf.\ also \cite{Bas}, \cite{BaS2}, \cite{Ham}, \cite{Kel},
\cite{Mil},
\cite{Nee}).
We briefly recall various
basic definitions and facts.
\begin{numba}
Suppose that $E$ and $F$ are real locally convex
spaces,
$U$ is an open subset of~$E$,
and $f\!: U\to F$ a map.
We say that $f$ is {\em of class $C^0$\/}
if it is continuous, and set $d^0f:=f$.
If~$f$ is continuous, we say
that~$f$ is {\em of class $C^1$\/}
if the (two-sided) directional derivative
$df(x,h):=\lim_{t\to 0}t^{-1}(f(x+th)-f(x))$
exists for all $(x,h)\in U\times E$
(where $t\in \R\take\{0\}$ with $|t|$ sufficiently small),
and the mapping $df\!: U\times E\to F$
is continuous.
Recursively, we define $f$
to be {\em of class~$C^k$\/} for $2\leq k\in\N$
if it is of class~$C^{k-1}$ and
$d^{k-1}f\!:U\times E^{2^{k-1}-1}\to F$ (having been defined
recursively) is a mapping of class~$C^1$
on the open subset $U\times E^{2^{k-1}-1}$ of the locally convex space
$E^{2^{k-1}}$. We then set $d^kf:=d(d^{k-1}f)\!:
U\times E^{2^k-1}\to F$.
The mapping~$f$
is called {\em smooth\/} or {\em of class
$C^\infty$\/} if it is of class
$C^k$ for all $k\in\N$.
\end{numba}
The preceding definition
of $C^r$-maps
(where $r\in \N_0\cup\{\infty\}$)
is particularly well-suited for
inductive arguments.
It is equivalent
to the usual definition of
$C^r$-maps
in the sense of Michal-Bastiani
(\cite{BCR}, p.\,24; \cite{RES}, Definition~1.8),
see~\cite{RES}, Lemma~1.14.\footnote{Note that
the ``iterated differentials'' $d^kf$ defined above
are denoted $D^kf$ in our main reference~\cite{RES},\linebreak
\hspace*{6 mm}whereas $d^kf$ has another meaning there.}
\begin{numba}
Since compositions of $C^r$-maps are of class~$C^r$
for $0\leq r\leq \infty$ 
(\cite{RES}, Proposition~1.15),
$C^r$-manifolds modelled on locally convex spaces
can be defined in the usual way,
using an atlas of charts with $C^r$-transition
functions. A {\em smooth Lie group\/}
is a group, equipped with a smooth manifold structure
modelled on a locally convex space,
with respect to which the group multiplication
and inversion are smooth mappings.
\end{numba}
\begin{numba}\label{pt1.3}
Let $X$ be a $C^r$-manifold (where $1\leq r\leq \infty$),
and $f\!:X\to E$ a mapping of class~$C^r$ into
a real locally convex space.
Then the tangent map $Tf\!: TX\to TE=E\times E$
has the form $(x,v)\mto (f(x),df(x,v))$
for $x\in X$ and $v\in T_xX$, where
$df:=\pr_2\circ Tf\!:TX\to E$.
We set $d^0f:=f$, $T^0X:=X$, and define
$d^kf\!:T^kX\to E$ recursively via $d^kf:=d(d^{k-1}f)$
for all $k\in \N$, $k\leq r$.
\end{numba}
Functions partially $C^r$ will
play an important role in our discussions.
\begin{numba}\label{defnpartCr}
Let $X$ be a Hausdorff topological space,
$Y$ be a $C^r$-manifold,
$f\!: U \to F$
be a mapping into a real locally convex space,
defined on an open subset $U\sub X\times Y$.
We say that
$f$ is {\em partially $C^r$ in the second argument\/}
if $f(x,\sbull)\!: V_x\to F$ is a mapping of class $C^r$
for all $x\in \pr_X(U)$ (where $V_x:=\{y\in Y\!:
(x,y)\in U\}$), and
the functions
$d_2^kf\!: \bigcup_{x\in \pr_X(U)}\{x\}\times T^k(V_x)\to F$,
defined via
$d_2^kf(x,\sbull):=d^k(f(x,\sbull))$
for $x\in \pr_X(U)$,
are continuous for all $k\in \N_0$, $k\leq r$.
Note that if $Y=E$ is a real locally
convex space, then
$d_2^kf$ simply is a mapping $U\times E^{2^k-1}\to F$.
\end{numba}
\begin{numba}\label{defncan}
Let $E$ and~$F$ be
complex locally convex spaces,
and~$U$ be an open subset of~$E$.
A function $f\!: U\to F$ is called
{\em complex analytic\/} or $\C$-analytic
if it is continuous
and for every $x\in U$,
there exists a $0$-neighbourhood~$V$ in~$E$
such that $x+V\sub U$ and
$f(x+h)=\sum_{n=0}^\infty \beta_n(h)$
for all $h\in V$
as a pointwise limit,
where $\beta_n\!: E\to F$ is
a continuous homogeneous polynomial over~$\C$
of degree~$n$, for each~$n\in\N_0$
(\cite{BaS2}, Defn.\ 5.6).
\end{numba}
\begin{numba}\label{dfcx}
If $F$ is sequentially complete
in the situation of
{\bf \ref{defncan}}, then~$f$ is complex analytic in the preceding sense
if and only if it is complex differentiable
on each affine line and continuous (\cite{BaS2}, Theorems 6.2 and 3.1).
In general, $f$ is complex analytic if
and only if it is smooth and $df(x,\sbull)\!:E\to F$
is complex linear for all $x\in U$
(\cite{RES}, Lemma~2.5).
\end{numba}
\begin{numba}\label{defnran}
Let $E$ and $F$ be real locally convex spaces,
$U$ be an open subset of~$E$, and
$f\!: U\to F$ be a mapping.
Following Milnor's lines, we
call~$f$ {\em real analytic\/}
or $\R$-analytic
if it extends to a complex analytic
mapping $V \to F_\C$
on some open neighbourhood $V$
of~$U$ in $E_\C$.
\end{numba}
\begin{numba}\label{complcomp}
Compositions of $\K$-analytic
mappings are $\K$-analytic
for $\K\in\{\R,\C\,\}$
(\cite{RES}, Proposition~2.7, Proposition~2.8).
Thus complex (analytic)
manifolds and real analytic manifolds,
as well as complex (analytic) Lie groups
and real analytic Lie groups modelled
on locally convex spaces
can be defined in the usual way.
\end{numba}
The following
observation will be useful later:
\begin{prop}\label{inittop}
Let $E$ and $F$ be locally convex
spaces, $U$ an open subset of $E$,
and $f\!: U\to F$ a map.
Suppose that the topology on $F$ is initial with
respect to a family $(\rho_i)_{i\in I}$
of continuous linear mappings
$\rho_i\!: F\to F_i$ into locally convex spaces~$F_i$,
and suppose that the embedding $\rho:=
(\rho_i)_{i\in I}\!:
F\to \prod_{i\in I} F_i$
has sequentially closed image.
Then $f$ is of class $C^r$
if and only if $f_i:=\rho_i\circ f\!: U\to F_i$
is of class $C^r$ for each $i\in I$.
\end{prop}
\begin{proof}
The necessity of the condition is apparent.
It is enough to prove
the sufficiency for $r\in \N_0$,
which we do by induction.
The case $r=0$ is clear, since $E$ carries the
initial topology. If $r\in \N$ and
if $\rho_i\circ f$ is of class $C^r$
for each~$i\in I$,
then $f$ is continuous by the preceding.
Clearly the $d(\rho\circ f)(x,v)$ exists
for $x\in U$, $v\in E$,
and is given by $d(\rho\circ f)(x,v)=(d(\rho_i\circ f)(x,v))_{i\in I}$.
As the differential quotient $d(\rho\circ f)(x,v)$
can be calculated by means of a sequence of difference
quotients and~$H$ is sequentially closed,
we have $d(\rho\circ f)(x,v)\in H$.
Thus $df(x,v)=(\rho|^H)^{-1}(d(\rho\circ f)(x,v))$ exists,
and $\rho_i\circ df=\pr_i\circ \rho\circ df
=\pr_i\circ d(\rho\circ f)=d(\rho_i\circ f)$
is a mapping of class $C^{r-1}$ for each $i\in I$.
By induction, $df$ is of class $C^{r-1}$
and so $f$ is of class $C^r$.
\end{proof}
\section{Functions spaces and mappings between them}
As a preliminary for our studies of spaces
of sections and mappings between such spaces,
we need to recall the definition
of various function spaces,
and study mappings between these.\vspace{2 mm}\\
{\bf Convention.\/}
Let $r\in \N_0\cup\{\infty\}$,
and $|r]:=\{n\in \N_0\!: n\leq r\}$.
If $1\leq r \leq \infty$,
we let $X$ be a
$C^r$-manifold in the sense of \cite{RES},
modelled on a locally convex space.
If $r=0$, we let~$X$ be any Hausdorff
topological space.
Throughout this section,
the symbols $r$, $|r]$, and $X$
will have the meanings just described.
We shall use the convention
$\infty + n:=\infty$, for
any $n\in \Z$. To unify our notation in the following definitions
and proofs, we define $T^0X:=X$
also in the case where~$X$ is a topological space,
and call continuous mappings on~$X$
also $C^0$-maps.\vspace{2 mm}\\
%
%
%
%
%The following definition of
%topologies on spaces of differentiable
%mappings
%is most convenient for
%our present purposes. Later on,
%we shall prefer alternative descriptions.
%
%
%
%
\begin{defn}\label{defntopfs}
If~$E$ is a locally convex topological vector space over~$\K$,
we let $C^r(X,E)$ be the $\K$-vector
space of $E$-valued $C^r$-mappings on~$X$.
Given
$\gamma\in C^{r}(X,E)$, we set $d^0\gamma:=\gamma$.
If $r>0$, then
$\gamma$ gives rise to
$C^{r-n}$-functions
$d^n\gamma \!: T^nX\to E$
for all $n\in |r]$ (see {\bf \ref{pt1.3}}).
We give $C^r(X,E)$ the topology which makes
\[
(d^n(\sbull))_{n\in |r]}\!: C^{r}(X,E)\to
\prod_{n\in |r]} C(T^n(X),E)_c,\;\;\;
\gamma\mto (d^n\gamma)_{n\in |r]}
\]
a topological embedding,
where $C(T^n(X),E)_c$
denotes $C(T^n(X),E)$,
equipped with the topology of compact convergence
(which coincides with the compact-open topology).
Given a compact subset $K$ of~$X$,
we define
\[
C^{\,r}_K(X,E):=\{
\gamma\in C^{r}(X,E)\!: \gamma|_{X\,\take\, K}=0\,\},
\]
and equip this closed vector subspace of
$C^r(X,E)$ with the induced topology.
If $X$ is a $\sigma$-compact finite-dimensional
$C^r$-manifold or $r=0$ and $X$ a $\sigma$-compact
locally compact space, we let $\cD^r(X,E)=\bigcup_{K} C^r_K(X,E)$,
equipped with the locally convex direct limit topology.
Occasionally, we shall write $\cD(X,E):=
\cD^\infty(X,E)$.
When $\K\in\{\R,\C\}$ is understood,
as usual we abbreviate $C^\infty(X):=
C^\infty(X,\K)$, $C^\infty_K(X):=C^\infty_K(X,\K)$,
and $\cD(X):=\cD^\infty(X,\K)$. 
\end{defn}
\begin{rem}\label{trivialI}
If $U$ is an open subset of $X$ in the preceding
situation, then obviously the restriction map
$\rho_U\!: C^r(X,E)\to C^r(U,E)$, $\gamma\mto \gamma|_U$
is continuous.
Furthermore, the topology on $C^r(X,E)$ is the initial
topology with respect to the family $(\rho_U)_{U\in \cU}$,
for every open cover $\cU$ of~$X$.
In fact, if $j\in |r]$ and $K$ is a compact subset
of $T^jX$, we find finitely many compact subsets
$C_1,\ldots, C_n$ of $X$ such that $C_i\sub U_i$
for some $U_i\in \cU$
and $\pi(K)\sub C_1\cup\cdots\cup C_n$,
where $\pi\!: T^jX\to X$, $T^j_x(X)\ni v \mto x$.
Thus $K=\bigcup_{i=1}^n K_i$,
where $K_i:=K\cap\pi^{-1}(C_i)\sub T^jU_i$.
The assertion easily follows.
\end{rem}
\begin{rem}\label{indu}
Note that if $X$ is a $\sigma$-compact
finite-dimensional manifold and $K$ is a compact subset of~$X$,
then we can find a sequence $(K_n)_{n\in\N}$ of compact
subsets of~$X$ such that $K_1=K$,
$K_n\sub \Int(K_{n+1})$,
and $\bigcup_{n\in\N}K_n=X$.
Clearly $\{K_n\!: n\in \N\}$ is a countable,
cofinal subset of the set of all compact
subsets of~$X$, directed under inclusion,
whence
$\cD^r(X,E)=\dl \, C^{\,r}_{K_n}(X,E)$,\vspace{-.4 mm}
a countable strict direct limit.
We deduce that
$\cD^r(X,E)$ induces the original topology
on $C^{\, r}_K(X,E)$, for every compact subset
$K\sub X$ (\cite{Bou}, Chapter~II,
\S4.6, Proposition~9\,(i)).
\end{rem}
The following fact (\cite{Glo}, Proposition~3.10),
which generalizes \cite{Nee}, Proposition III.7,
will be useful later. When $r>0$,
we assume $\dim(X)<\infty$ here to exclude
trivialities.
\begin{prop}\label{pushforw0}
Given $r$ and $X$ as above,
let~$E$
and~$F$ be locally convex spaces,
$U$ be an open subset of~$E$,
$K$ be a compact subset of~$X$,
$k\in \N_0\cup\{\infty\}$,
and $f\!: X\times U \to F$ be a mapping
such that
\begin{itemize}
\item[\n (a)]
$f(x,0)=0$ for all $x\in X\,\take\, K$,
\item[\n (b)]
$f$ is partially $C^k$ in the second argument
$($see {\bf \ref{defnpartCr}}$)$, and
\item[\n (c)]
The functions $d_2^jf$ are of
class $C^r$, for every $j\in |k]$.
\end{itemize}
If $K\not=X$, we assume that $0\in U$.
Then $C^{\, r}_K(X,U):=\{\gamma\in C^{\, r}_K(X,E)\!:
\im(\gamma)\sub U\}$ is an open subset
of $C^{\, r}_K(X,E)$, and
\[
f_*\!: C^{\,r}_K(X,U)\to C^{\,r}_K(X,F),\;\;\;\;
\gamma\mto f\circ (\id_X,\gamma)
\]
is a mapping of class~$C^k$.\Punkt
\end{prop}
Our main results will be analogues of the preceding
proposition
in the setting of spaces of sections.
The following variant,
valid even for infinite-dimensional~$X$,
is vital
for our later constructions:
\begin{prop}\label{pushforw1}
Given $r$ and $X$ as above,
let $E$ and~$F$ be locally convex spaces,
$K$ be a compact subset of~$X$,
$k\in \N_0\cup\{\infty\}$,
and $f\!: X\times E  \to F$ be a mapping
such that
\begin{itemize}
\item[\n (a)]
$f$ is partially $C^k$ in the second argument,
and
\item[\n (b)]
The functions $d_2^jf$ are of
class $C^r$, for every $j\in |k]$.
\end{itemize}
Then
\[
f_*\!: C^{\,r}(X,E)\to C^{\,r}(X,F),\;\;\;\;
\gamma\mto f\circ (\id_X,\gamma)
\]
is a mapping of class~$C^k$.
\end{prop}
\begin{proof}
We remark first that
as $f$ is a mapping of class~$C^r$ by hypothesis\,(b),
the Chain Rule shows that
$f_*\gamma$ is of class~$C^r$, for every
$\gamma\in C^r(X,E)$.\vspace{1.7 mm}\\
It clearly suffices to prove the proposition for
$k<\infty$. We prove the following claim by induction
on~$0\leq j\leq k$, thus establishing the assertion:\vspace{1.7 mm}\\
{\bf Claim.\/} {\em The mapping $f_*$ is of class $C^j$,
and $d^{j}(f_*)=(d_2^{j}f)_*$, where we identify
$C^r(X,E)^{2^j}$ with
$C^r(X, E^{2^j})$ by means of\/}
\cite{Glo}, {\em Lemma\/}~3.4.\vspace{1.7 mm}\\
{\em Case $j=0$\/}:
In this case, we only need to show
that~$f_*$ is continuous.
Recall that $T^0X:=X$, $T^0E:= E$,
and $d^0f:=f$.
Given $\gamma\in C^r(X, E)$,
furthermore $T^0\gamma:=\gamma$.
If $r>0$, 
we have
\[
T(f_*\gamma) = T(f\circ (\id_X,\gamma)) =Tf\circ (\id_{TX}, T\gamma)
\]
using the identification
$T(X\times E)\isom T(X)\times T(E)$,
and therefore $d(f_*\gamma)=df\circ (\id_{TX}, T\gamma)=(df)_*(T\gamma)$.
Inductively, we obtain
\begin{equation}\label{eqndns}
d^n(f_*\gamma)=d^nf\circ (\id_{T^nX}, T^n\gamma)=(d^nf)_*(T^n\gamma)
\end{equation}
for all $n\in |r]$.
Note that $(d^nf)_*\!: C(T^nX, T^nE)_{c.o.}\to C(T^nX,F)_{c.o.}$ 
is continuous by \cite{Glo}, Lemma~3.9,
and $C^r(X,E)\to C(T^nX, T^nE)_{c.o.}$,
$\gamma\mto T^n\gamma$ is continuous by {\em loc.\,cit.}\
Lemma~3.8.\footnote{Strictly speaking, all of the cited lemmas
are formulated in \cite{Glo}
for finite-dimensional~$X$ only, but they\linebreak
\hspace*{5.5 mm} remain valid for infinite-dimensional~$X$, by exactly the same
proof.}
We deduce that
$\Phi_n\!: C^r(X,E)\to C(T^nX,F)_{c.o.}$,
$\gamma\mto d^n(f_*\gamma)$
is a continuous mapping.
The topology on
$C^r(X,F)$ being the initial topology with respect
to the mappings $g\mto d^ng\in C(T^nX,F)_{c.o.}$,
where $n\in |r]$,
we infer from the preceding that $f_*$ is continuous.\vspace{2.5 mm}\\
{\em Induction Step.\/}
Suppose that $f_*$ is of class $C^j$ (where $0\leq j<k$)
and $d^{j}(f_*)=(d^j_2f)_*$.
We have to show that $d^j(f_*)$ is of class $C^1$,
with $d(d^jf_*)=(d^{j+1}_2f)_*$.
To this end, let
$\gamma,\eta \in C^r(X, E^{2^j})$
be given.
Then
$G\!:
\R \times X\to F$,
$G(s,x):=d_2^{j+1}\!f\,(x,\gamma(x)+s\eta(x),\eta(x))$
is a mapping of class~$C^r$, with $G(0,\sbull)=(d_2^{j+1}f)_*(\gamma,\eta)$,
and
\[
H(h,x):=\frac{1}{h}[d^j_2f(x,(\gamma+h\eta)(x))-d^j_2f(x,\gamma(x))]=
\int_0^1 G(th,x)\, dt\]
for all $x\in X$ and $0\not=h\in \R$
by the Fundamental Theorem of Calculus
(\cite{RES}, Theorem~1.5).
For any $n\in |r]$, $y\in T^nX$, and $0\not=h\in \R$,
we obtain
\[
d^n_2H(h,y)=
\int_0^1 d^n_2G(th,y)\, dt,\]
where $d^n_2G\!:\; \R \times\, T^nX\to F$
is a mapping of class $C^{r-n}$ and thus continuous.
Now suppose that~$C$ is a compact subset of $T^nX$.
We deduce from the uniform continuity of
the mapping $[0,1]\times [-1,1 ]\times C\to F$,
$(t,h,y)\mto d^n_2G(th,y)$ that
\[
\lim_{h\to 0}
d_2^nH(h,y)= 
\lim_{h\to 0}\int_0^1 d^n_2G(th,y)\, dt
= d^n_2G(0,y)
= d^n\!\!\left((d_2^{j+1}f)_*(\gamma,\eta)\right)(y)
\]
uniformly for~$y\in C$.
Thus
$
\lim_{h\to 0}\frac{1}{h}[(d^j_2f)_*(\gamma+h\eta)-(d^j_2f)_*(\gamma)]
=(d_2^{j+1}f)_*(\gamma,\eta)$
in $C^{\,r}(X,F)$.
We have shown that the directional
derivative $d(d^jf_*)(\gamma,\eta)$
exists, and is given by $(d^{j+1}_2f)_*(\gamma,\eta)$.
Replacing~$f$ by
$d^{j+1}_2f$ (which satisfies analogous
hypotheses, $k$ being replaced by $k-j-1$),
the case $j=0$ of the present induction
shows that $d^{j+1}(f_*)=d(d^jf_*)=(d_2^{j+1}f)_*$ is continuous.
This completes the proof.
\end{proof}
\begin{cor}\label{functorial1}
If
$f\!: E\to F$ is a smooth
mapping between locally convex spaces,
then also the mapping
\[
C^r(X,f)\!:
C^r(X,E)\to C^r(X,F),\;\;\;\;\;
\gamma\mto f\circ \gamma
\]
is smooth.
\end{cor}
\begin{proof}
We set $g(x,y):= f(y)$ for
$(x,y)\in X\times E$.
Then $g$ is partially $C^\infty$ in the second
argument, with $d_2^ng(x,y)=d^nf(y)$
for $x\in X$, $y\in T^nE= E^{2^n}$,
which is a smooth function of~$(x,y)$.
By Proposition~\ref{pushforw1},
$C^r(X,f)=g_*$
is smooth.
\end{proof}
\begin{cor}\label{contmlt}
Let $E$ be a locally convex $\K$-vector space
and $f\!:X\to \K$ be a $C^r$-map.
Then
\[
m_f\!: C^r(X,E)\to C^r(X,E),\;\;\;\;
\gamma\mto f\cdot \gamma
\]
$($pointwise product$)$
is a continuous $\K$-linear map.
\end{cor}
\begin{proof}
The mapping $g\!: X\times E\to E$,
$g(x,v):=f(x)\cdot v$
is a composition $g=s\circ (f\times \id_E)$,
where scalar multiplication $s\!:\K\times E\to E$
is smooth being a continuous bilinear map,
$f$ is of class~$C^r$, and $\id_E$ is smooth.
Thus $g$ is of class~$C^r$.
By Proposition~\ref{pushforw1},
$m_f=g_*$ is continuous.
\end{proof}
\section{Vector bundles and spaces of sections}\label{sec3}
In this section, we define
vector bundles (of locally convex spaces)
and provide some background material
concerning spaces of sections
and their topologies.
\begin{defn}\label{defnbdle}
Let $X$ be a $C^r$- (resp., $\K$-analytic) manifold,
modelled over a locally convex $\K$-vector space~$Z$,
and $F$ be a locally convex $\bL$-vector space,
where $\bL\in\{\R,\C\}$
and $\K\in \{\R,\bL\}$.
A $C^r$-\,(resp., $\K$-analytic)
{\em vector bundle\/} over $X$, with typical fibre~$F$,
is a $C^r$-\,(resp., $\K$-analytic) manifold~$E$,
together with a $C^r$-\,(resp., $\K$-analytic)
surjection $\pi\!: E\to X$ and equipped with an
$\bL$-vector space structure on each fibre
$E_x:=\pi^{-1}(\{x\})$,
such that for each $x_0\in X$ there exists
an open neighbourhood $X_\psi$ of~$x_0$ in~$X$
and a $C^r$-\,(resp., $\K$-analytic)
diffeomorphism
\[
\psi\!: \pi^{-1}(X_\psi)\to X_\psi\times F\]
(called a ``local trivialization of $E$ about $x_0$'')
such that
$\psi(E_x)=\{x\}\times F$ for each $x\in X_\psi$
and $\pr_F\circ \psi|_{E_x}\!: E_x\to
F$ is $\bL$-linear
(and thus an isomorphism of topological vector spaces
with respect to the topology on $E_x$ induced by~$E$).
\end{defn}
\begin{rem}\label{chnge}
In the preceding situation,
given two local trivializations
$\psi\!: \pi^{-1}(X_\psi)\to X_\psi\times F$
and $\phi\!: \pi^{-1}(X_\phi)\to X_\phi\times F$,
we have $\phi(\psi^{-1}(x,v))=(x, g_{\phi,\psi}(x).v)$
for some function $g_{\phi,\psi}\!: X_\phi\cap X_\psi\to
\GL(F)\sub L(F,F)_b$ (the space of continuous
linear self-maps, equipped with the topology of
uniform convergence on bounded subsets of~$F$).
Then $G_{\phi,\psi}\!:(X_\phi\cap X_\psi)\times F\to F$,
$(x,v)\mto g_{\phi,\psi}(x).v$
is a $C^r$- (resp., $\K$-analytic) mapping (since $\phi\circ \psi^{-1}$
is so),
%
% MORE PRECISELY:
%(\psi^{-1}|_{(X_\phi\cap X_\psi)\times F}^{\pi^{-1}(X_\phi\cap X_\psi)})$
%
and it can be shown that $g_{\phi,\psi}$
is a mapping of class $C^{r-2}$ (resp., $\K$-analytic)
automatically~\cite{HYP}. In the $C^r$-case, when $r<\infty$,
the reader might wish to include as an additional axiom
in the definition of a vector bundle that
the mappings $g_{\phi,\psi}$ be of class $C^r$
(although we shall not do so here).
All of our arguments remain valid
in this more restrictive setting.
\end{rem}
In the present context, we are only interested in
$C^r$-vector bundles and their sections.
\begin{defn}
A {\em $C^r$-section\/} of a $C^r$-vector bundle
$\pi\!: E\to X$ is a $C^r$-mapping $\sigma\!:
X\to E$ such that $\pi\circ \sigma=\id_X$.
Its {\em support\/} $\Supp(\sigma)$
is the closure of $\{x\in X\!:
\sigma(x)\not=0_x \}$.
We let $C^r(X,E)$, $C^r_K(X,E)$, resp.,
$\cD^r(X,E)$ be the sets of all
$C^r$-sections of $E$, the set of all $C^r$-sections with support
contained
in a given compact subset $K\sub X$,
resp., the set of all compactly supported $C^r$-sections.
We abbreviate $\cD(X,E):=\cD^\infty(X,E)$.
\end{defn}
Making use of scalar multiplication
and addition in the individual fibres,
we obtain natural vector space structures
on $C^r(X,E)$, $C^r_K(X,E)$, and $\cD^r(X,E)$.
The zero-element is
the {\em zero-section\/} $0_\sbull\!:
X\to E$, $x\mto 0_x\in E_x$.
\begin{rem}\label{notinterest}
Note that when $C^r_K(X,E)$ or $\cD^r(X,E)$
is a non-trivial vector space,
then $X$ has to be finite-dimensional.
\end{rem}
\begin{defn}
If $\pi\!: E\to X$ is a $C^r$-vector bundle
with typical fibre~$F$,
$\sigma\!:X\to E$ a $C^r$-section,
and $\psi\!:\pi^{-1}(X_\psi)\to X_\psi\times F$ a local
trivialization, we define
$\sigma_\psi:=\pr_F\circ \psi\circ \sigma|_{X_\psi}\!: X_\psi\to F$.
Thus $\psi(\sigma(x))=(x,\sigma_\psi(x))$
for all $x\in X_\psi$.
\end{defn}
Note that $\sigma_\psi$ is a mapping of class~$C^r$ here.
The symbols $g_{\phi,\psi}$,
$G_{\phi,\psi}$, and $\sigma_\psi$
will always be used with the meanings just described,
without further explanation.
\begin{defn}
If $\pi\!: E\to X$ is a vector bundle
and $\cA$ a set of local trivializations $\psi$ of~$E$
whose domains cover~$E$,
then we call $\cA$ an {\em atlas\/} of local trivializations.
\end{defn}
\begin{la}\label{famgivessec}
If $\pi\!: E\to X$ is a $C^r$-vector bundle
with typical fibre~$F$, and $\cA$ an atlas
of local trivializations for~$E$,
then
\[
\Gamma\!: C^r(X,E)\to \prod_{\psi\in\cA} C^r(X_\psi,F),\;\;\;\;
\sigma\mto (\sigma_\psi)_{\psi\in\cA}
\]
is an injection, whose image is the closed vector subspace
$H:=\{(f_\psi)\in \prod C^r(X_\psi,F)\!: \;(\forall \phi,\psi\in\cA,\;
\forall x\in X_\phi\cap X_\psi)\;
f_\phi(x)=g_{\phi, \psi}(x).f_\psi(x)\,\}$
of $\, \prod_{\psi\in\cA} C^r(X_\psi,F)$.
\end{la}
\begin{proof}
It is obvious that $\im\,\Gamma\sub H$,
and clearly~$\Gamma$ is injective.
If now $(f_\psi)_{\psi\in\cA}\in H$,
we define $\sigma\!: X\to E$ via
$\sigma(x):=\psi^{-1}(x,f_\psi(x))$ if $x\in X_\psi$.
By definition of~$H$, $\sigma(x)$ is independent of the choice of~$\psi$.
As $\psi\circ \sigma|_{X_\psi}=(\id_{X_\psi},f_\psi)$,
the mapping $\sigma\!: X\to E$ is of class~$C^r$.
Thus $\sigma$ is a $C^r$-section,
and $\Gamma(\sigma)=(f_\psi)_{\psi\in\cA}$
by definition of~$\sigma$.
We deduce that $\im\,\Gamma=H$.
The closedness of~$H$ follows
from the continuity of the restriction mapping
$C^r(X_\psi,F)\to C^r(X_\phi\cap X_\psi,F)$
and the continuity of $(G_{\phi,\psi})_*\!:
C^r(X_\phi\cap X_\psi,F)\to C^r(X_\phi\cap X_\psi,F)$
(Theorem~\ref{pushforw1}),
for all $\phi,\psi\in\cA$.
\end{proof}
\begin{defn}\label{dfntops}
Let $\pi\!:E\to X$ be a $C^r$-vector bundle, with typical fibre~$F$,
and $\cA$ be the atlas of all local trivializations
of~$E$.
We give $C^r(X,E)$ the locally convex vector topology making the
linear mapping
\[
\Gamma\!: C^r(X,E)\to\prod_{\psi\in\cA}C^r(X_\psi, F),\;\;\;
\sigma\mto (\sigma_\psi)_{\psi\in\cA}\]
a topological embedding.
We give $C^r_K(X,E)$ the topology induced by $C^r(X,E)$,
for every compact subset $K\sub X$.
If~$X$ is finite-dimensional and
$\sigma$-compact, we
equip $\cD^r(X,E)=\bigcup_K C^r_K(X,E)$ with
the vector topology making
it the direct limit of its subspaces $C^r_K(X,E)$ in the
category of locally convex spaces.
It induces the original topology
on each $C^{\,r}_K(X,E)$
(cf.\ Remark~\ref{indu}).
\end{defn}
By the preceding definition,
the topology on $C^r(X,E)$ is initial
with respect to the family
$(\theta_\psi)_{\psi\in\cA}$,
where $\theta_\psi\!: C^r(X,E)\to C^r(X_\psi,F)$,
$\sigma\mto \sigma_\psi$.\vspace{2 mm}\\
Two simple lemmas will be useful
later:
\begin{la}\label{anyatlas}
The topology on $C^r(X,E)$ described in Definition~{\n \ref{dfntops}\/}
is initial with respect to $(\theta_\psi)_{\psi\in \cB}$,
for any atlas $\cB\sub \cA$ of local trivializations
for~$E$.
\end{la}
\begin{proof}
We let $\cO$ be the initial topology
on $C^r(X,E)$ with respect to $(\theta_\psi)_{\psi\in \cB}$,
which apparently is coarser than initial topology
with respect to $(\theta_\psi)_{\psi\in\cA}$.
Fix a local trivialization
$\phi\in \cA$.
Then $\{X_\phi\cap X_\psi\!:\psi\in \cB\}$
is an open cover for $X_\phi$.
In view of Remark~\ref{trivialI}, $\theta_\phi$
will be continuous on $(C^r(X,E),\cO)$ if we can show
that $\sigma\mto \theta_\phi(\sigma)|_{X_\phi\cap X_\psi}$
is continuous for all $\psi\in \cB$.
But, with $G_{\phi,\psi}$
as in Remark~\ref{chnge}, the latter mapping is the
composition of $(G_{\phi,\psi})_*\!: C^r(X_\phi\cap X_\psi,F)\to
C^r(X_\phi\cap X_\psi,F)$ (which is continuous
by Theorem~\ref{pushforw1})
and $C^r(X,E)\to C^r(X_\phi\cap X_\psi,F)$,
$\sigma\mto \theta_\psi(\sigma)|_{X_\phi\cap X_\psi}$,
which is continuous by Remark~\ref{trivialI}.
Thus, $\theta_\phi$
being continuous on $(C^r(X,E),\cO)$
for each $\phi\in\cA$,
the topology~$\cO$ is finer than the initial
topology with respect to the family $(\theta_\phi)_{\phi\in \cA}$,
which completes the proof.
\end{proof}
If $\pi\!: E\to X$ is a $C^r$-vector bundle
and~$U$ an open subset of~$X$, then $\pi|_{E_U}^U\!: E|_U\to U$
makes the open submanifold $E|_U:=\pi^{-1}(U)$ of~$E$ a $C^r$-vector bundle
over the base~$U$.
\begin{la}\label{restrcts}
In the situation of Definition~{\n \ref{dfntops}},
let~$U$ be an open subset of~$X$,
and $K$ be a compact subset of~$U$.
Then the restriction map
$C^r(X,E)\to C^r(U,E|_U)$,
$\sigma\mto \sigma|_U$ is continuous,
and
the corresponding restriction map
$\rho_U\!: C^r_K(X,E)\to C^r_K(U,E|_U)$ is
an isomorphism of topological vector spaces.
\end{la}
\begin{proof}
The first assertion is obvious.
Thus $\rho_U$ is continuous as well,
and apparently it is a linear bijection.
To see that $\rho_U$ is an isomorphism
of topological vector spaces,
we let $\cA_0$ be an atlas of local trivializations
for $E|_{X\,\take\, K}$,
and $\cA_1$ be an atlas for $E|_U$.
Then the topology
on $C^r_K(U,E|_U)$ is initial
with respect to the family of mappings
$\theta_\psi^U\!:C^r_K(U,E|_U)\to C^r(X_\psi,F)$,
$\sigma\mto \sigma_\psi$, for $\psi\in \cA_1$.
Furthermore, by Lemma~\ref{anyatlas},
the topology on $C^r_K(X,E)$ is
initial with respect to
family of mappings $\theta^X_\psi\!:
C^r_K(X,E)\to C^r(X_\psi,F)$
for $\psi\in \cA_0\cup\cA_1$
(defined analogously).
As $\theta^Y_\psi\circ \rho_U=\theta^X_\psi$
for all $\psi\in \cA_1$ whereas $\theta^X_\psi=0$
for all $\psi\in \cA_0$,
the assertion easily follows.
\end{proof}
If $\pi\!: E\to X$ is a $C^r$-vector bundle
over a finite-dimensional base~$X$,
and $K$ a compact subset of~$X$,
then there exists a $\sigma$-compact, open neighbourhood
$U$ of $K$ in~$X$. By the preceding lemma
we have $C^r_K(X,E)\isom C^r_K(U,E|_U)$.
Therefore, no generality will be lost
by formulating all results concerning $C^r_K(X,E)$
for $\sigma$-compact~$X$.
\begin{numba}\label{sumsvble}
If $\pi_j\!: E_j\to X$ are $C^r$-vector bundles
with fibre $F_j$
for $j\in\{1,2\}$,
over the same base~$X$, then $E_1\oplus E_2:=
\bigcup_{x\in X} (E_1)_x\times (E_2)_x$
is a $C^r$-vector bundle over $X$ in a natural way,
with projection $\pi\!: E_1\oplus E_2\to X$, $v\mto x$
if $v\in (E_1)_x\times
(E_2)_x $. Let $\phi_j\!: \pi_j^{-1}(X_{\phi_j})\to X_{\phi_j}\times F_j$
be local trivialization of~$E_j$ for $j\in \{1,2\}$,
and
$X_{\phi_1\oplus \phi_2}:=X_{\phi_1}\cap X_{\phi_2}$.
Then
\[
\phi_1\oplus \phi_2\!: \pi^{-1}(X_{\phi_1\oplus\phi_2}) \to
X_{\phi_1\oplus\phi_2}\times (F_1\times F_2),
\]
$(\phi_1\oplus\phi_2)(v,w) :=(x, \pr_{F_1}(\phi_1(v)),
\pr_{F_2}(\phi_2(w)))$
for $(v,w)\in (E_1)_x\times (E_2)_x$,
is a local trivialization of~$E_1\oplus E_2$.
\end{numba}
\begin{numba}
It is easy to see that the linear mapping
\[
C^r(X,E_1)\times C^r(X,E_2)\to C^r(X,E_1\oplus E_2),\;\;\;
(\sigma_1,\sigma_2)\mto (x\mto (\sigma_1(x),\sigma_2(x)))
\]
is an isomorphism of topological vector spaces,
and so is the corresponding linear map $\cD^r(X,E_1)\times \cD^r(X,E_2)\to
\cD^r(X,E_1\oplus E_2)$.
\end{numba}
\begin{numba}
If $U$ is an open subset of $E_1$,
we define
\[
U\oplus E_2:=\bigcup_{x\in X} (U\cap (E_1)_x)\times (E_2)_x.
\]
It is easily verified that $U\oplus E_2$ is an open subset
of $E_1\oplus E_2$.
\end{numba}
\begin{numba}
Given a vector bundle $E$,
and $k\in \N$, we abbreviate $E^k:=E\oplus\cdots \oplus E$
($k$-fold sum).
\end{numba}
If $F$ is a locally convex real or complex
topological vector space,
we shall say that a zero-neighbourhood $U\sub F$
is $\R$-{\em balanced\/}
if it is balanced in~$F$, considered as a real
locally convex space.
\begin{defn}
Given a $C^r$-vector bundle
$\pi\!: E\to X$ over a finite-dimensional,
$\sigma$-compact base, and a neighbourhood
$U$ of the zero-section $0_X$ in~$E$,
we set
\[
C^{\,r}_K(X,U):=\{\sigma\in C^{\,r}_K(X,U)\!:\im(\sigma)\sub U\,\}
\]
for compact subsets $K$ of~$X$,
and define
$\cD^r(X,U):= \{\sigma\in \cD^r(X,E)\!:
\im(\sigma)\sub U\,\}$.
The neighbourhood $U$ will be called
$\R$-{\em balanced\/}
if $U\cap E_x$ is an $\R$-balanced
zero-neighbourhood in~$E_x$, for each $x\in X$.
\end{defn}
In the preceding situation, we have:
\begin{la}\label{thingsopen}
If~$U$ is an $\R$-balanced open neighbourhood
of $0_X$ in~$E$, then
$C^{\,r}_K(X,U)$ is an open zero-neighbourhood in
$C^{\, r}_K(X,E)$, and $\cD^r(X,U)$ is
an open zero-neighbourhood in $\cD^r(X,E)$.
\end{la}
\begin{proof}
Since
$\cD^r(X,U)\cap C^{\,r}_K(X,E)
=C^{\, r}_K(X,U)$ and
$C^{\,r}_K(X,E)\sub \cD^r(X,E)$ carries the subspace
topology,
it suffices to show
that $\cD^r(X,U)$ is an open zero-neighbourhood
in $\cD^r(X,E)$.
To this end,
let $\tau \in \cD^r(X,U)$;
we have to show that $\cD^r(X,U)$ is
a neighbourhood of~$\tau$.
For $x\in X$, choose a local trivialization
$\psi_x\!:\pi^{-1}(B_x)\to B_x\times F$
of~$E$ about~$x$.
Then $Q:=\psi_x(U\cap \pi^{-1}(B_x))$ is an open neighbourhood
of $(x,\tau_{\psi_x}(x))$ in $B_x\times F$. Since
$U\cap E_x$ is
$\R$-balanced, indeed $C:=\{x\}\times ([-1,1]\tau_{\psi_x}(x))
\sub Q$. As $C$ is compact, we find an open neighbourhood
$A_x\sub B_x$ of~$x$ and an open, $\R$-balanced,
convex
zero-neighbourhood
$H_x$ in~$F$ such that $A_x\times I_x\sub Q$,
where $I_x:=([-1,1]\tau_{\psi_x}(x))+H_x$.
Note that $I_x$
is an $\R$-balanced, convex, open zero-neighbourhood in~$F$
containing $\tau_{\psi_x}(x)$. After shrinking
$A_x$, we may assume that $A_x$
is relatively compact,
$\wb{A_x}\sub B_x$,
and
$\tau_{\psi_x}(\wb{A_x})\sub I_x$.
There is an open cover $(U_x)_{x\in X}$ of~$X$
such that $(\wb{U}_{\!x})_{x\in X}$
is locally finite, and
$\wb{U}_{\!x}\sub A_x$ for all $x\in X$.
Since $\wb{U}_{\!x}$ is a compact subset of~$B_x$
and $I_x\sub F$ an open zero-neighbourhood,
$\{\gamma\in C^r(B_x,F)\!:\gamma(\wb{U}_{\!x})\sub I_x\}$
is an open zero-neighbourhood
in $C^r(B_x,F)$.
The mapping
$\theta_{\psi_x}\!:\cD^r(X,E)\to C^r(B_x,F)$,
$\sigma\mto \sigma_{\psi_x}$
being continuous,
we deduce that
\[
V_x:=\{\sigma\in \cD^r(X,E)\!: \sigma_{\psi_x}(\wb{U}_{\!x})\sub
I_x\}
\]
is an open, convex zero-neighbourhood
in $\cD^r(X,E)$.
Note that $V_x=\cD^r(X,E)$ if $U_x=\emptyset$.
Then
$V:=\bigcap_{x\in X}V_x$
is a convex subset of $\cD^r(X,E)$.
By construction, we have $\tau\in V$
and $V\sub \cD^r(X,U)$.
So, the proof will be complete
if we can show that~$V$ is an open
zero-neighbourhood in $\cD^r(X,E)$.
To see this,
note that, for every compact subset~$K$
of~$X$, the set $F_K:=\{x\in X\!: \wb{U}_{\!x}\cap K\not=\emptyset\}$
is finite, whence
$C^{\,r}_K(X,E)\cap V=C^{\,r}_K(X,E)\cap
\bigcap_{x\in F_K}V_x$
is a convex, open zero-neighbourhood
in $C^{\,r}_K(X,E)$.
We deduce that~$V$ is a convex, open zero-neighbourhood
in the locally convex direct limit
$\cD^r(X,E)=\dl\, C^{\,r}_K(X,E)$.\vspace{-.8 mm}
\end{proof}
\section{Push-forward of {\boldmath $C^r$- and
$C^r_K\,$-\,sections}}\label{sec4}
In this section and the next,
we prove generalizations of Proposition~\ref{pushforw0}
and Proposition~\ref{pushforw1}
for mappings between spaces of sections in vector
bundles.
\begin{defn}\label{defnparvect}
Suppose that $\pi_1\!:E_1\to X$ and $\pi_2\!: E_2\to X$
are $C^r$-vector bundles over the same base,
with typical fibres $F_1$ and $F_2$, respectively.
A mapping $f\!: U \to E_2$,
defined on an open neighbourhood~$U$
of the zero-section~$0_X$ in~$E_1$,
is called a {\em bundle map\/}
if it preserves fibres, {\em i.e.},
$f(U\cap (E_1)_x)\sub (E_2)_x$ for all $x\in X$.
Then, given local trivializations
$\psi\!: \pi_1^{-1}(X_\psi)\to X_\psi\times F_1$
and $\phi\!:\pi_2^{-1}(X_\phi)\to X_\phi\times F_2$
of $E_1$, resp., $E_2$, we have
\[
\phi(f(\psi^{-1}(x,v)))=(x, f_{\phi,\psi}(x,v))
\]
for all $(x,v)\in U_{\phi,\psi}:=\psi(U\cap E_1|_{X_\psi\cap X_\phi})
\sub (X_\psi\cap X_\phi)\times F_1$,
for a uniquely determined mapping
\[
f_{\phi,\psi}\!: U_{\phi,\psi}\to F_2.
\]
Given $k\in \N_0\cup \{\infty\}$,
we shall say that a bundle map $f\!: E_1\supseteq U \to E_2$ is
{\em partially $C^k$ in the vector variable\/}
if $f_{\phi,\psi}$ is partially $C^k$ in the second variable
({\em i.e.}, in $v$), for any pair $(\psi,\phi)$
of local trivializations of $E_1$ and $E_2$
(equivalently, for any $(\psi,\phi)\in\cA_1\times\cA_2$,
where $\cA_i$ is any (not necessarily maximal)
atlas for~$E_i$).
\end{defn}
\begin{rem}
For our applications
in Section~\ref{secmapgps},
it would not be enough
to restrict our studies to the
case $U=E_1$,
where $f$ is defined globally.
\end{rem}
\begin{numba}
In the situation of Definition~\ref{defnparvect},
suppose that $f\!: U \to E_2$ is a bundle map
which is partially $C^k$ in the vector variable.
Then, as we shall presently see, $f$ gives rise to
a bundle map
\[
\delta f\!: U \oplus E_1\to E_2
\]
on $U\oplus E_1\sub E_1\oplus E_1$,
partially $C^{k-1}$ in the vector variable,
which is an analogue
of the mapping $d_2g$ for an
ordinary vector-valued function~$g$ which is partially~$C^k$
in the second variable.
It is determined by the condition that
\begin{equation}\label{chreqn}
(\delta f)_{\phi, (\psi \oplus \psi)}
=d_2 (f_{\phi,\psi})
\end{equation}
for all
local trivializations
$\psi\!: \pi_1^{-1}(X_\psi)\to X_\psi\times F_1$
of $E_1$ and $\phi\!: \pi_2^{-1}(X_\phi)\to X_\phi\times F_2$
of $E_2$.
\end{numba}
\begin{numba}
To see that a bundle map $\delta f\!: U \oplus E_1\to E_2$
satisfying (\ref{chreqn}) exists, we have to show
that, given any local trivializations
$\psi$, $\psi'$ of $E_1$ and $\phi$, $\phi'$ of $E_2$,
we have
\begin{equation}\label{eq4}
d_2 f_{\phi,\psi}(x,v,w)=
g_{\phi,\phi'}(x)\, . \,
d_2f_{\phi',\, \psi'}(x,g_{\psi',\psi}(x).v,\, g_{\psi',\psi}(x).w) 
\end{equation}
for all $(x,v,w)\in (U_{\phi,\psi}\times F_1)
\,\cap\, ((X_{\phi'}\cap X_{\psi'})\times F_1\times F_1)$.
As
\[
f_{\phi,\psi}(x,v)
=g_{\phi,\phi'}(x).f_{\phi',\psi'}(x,g_{\psi',\psi}(x).v)
\]
for all $(x,v)\in U_{\phi,\psi}\,\cap\,
(X_{\phi'}\cap X_{\psi'})\times F_1$,
the identity\,(\ref{eq4}) is an immediate consequence of the
chain rule.
\end{numba}
\begin{numba}
If the bundle map $f\!: E_1\supseteq U\to E_2$ is partially
$C^k$ in the vector variable,
we obtain bundle maps $\delta^jf \!: U\oplus
E^{2^j-1}_1\to E_2$
for $j\in |k]$
(which are partially $C^{k-j}$ in the vector variable)
recursively via
$\delta^0f:= f$, $\delta^j f:= \delta(\delta^{j-1}f)$
(for $0<j\in |k]$).
Note that the mappings $(d_2)^j f_{\phi,\psi}$
are of class $C^r$ for all $j\in |k]$
if and only if the
the bundle mappings $\delta^j f\!: U\oplus E_1^{2^j-1}\to E_2$
are of class~$C^r$ for all $j\in |k]$.
\end{numba}
\begin{thm}\label{pushforw2}
Suppose that $\pi_1\!: E_1\to X$ and
$\pi_2\!: E_2\to X$ are $C^r$-vector bundles
over the same $($not necessarily finite-dimensional$)$
base, with typical fibres $F_1$,
resp., $F_2$.
Let $k\in \N_0\cup\{\infty\}$,
and suppose that $f\!: E_1\to E_2$ is a bundle
map such that
\begin{itemize}
\item[\n (a)]
$f$ is partially $C^k$ in the vector variable, and
\item[\n (b)]
The functions $\delta^jf$ are of
class $C^r$, for every $j\in |k]$.
\end{itemize}
Then
\[
C^r(X,f)\!: C^r(X,E_1)\to C^r(X,E_2),\;\;\;
\sigma \mto f\circ \sigma
\]
is a mapping of class~$C^k$,
with differentials
$d^jC^r(X,f)=C^r(X, \delta^j f)$
for all $j\in |k]$,
using the identification $C^r(X,E_1)^{2^j}\isom
C^r(X,E^{2^j}_1)$.
\end{thm}
\begin{proof}
We let $\cP$ be an open cover of
$X$ such that for every $P\in\cP$,
there are local trivializations
$\psi_P\!:\pi_1^{-1}(P)\to P\times F_1$
and $\phi_P\!: \pi_2^{-1}(P)\to P\times F_2$.
Given $P\in \cP$,
we have
\[
\beta_P\circ C^r(X,f)= (f_{\phi_P,\psi_P})_* \circ
\alpha_P
\]
where $\alpha_P\!: C^r(X,E_1)\to C^r(P,F_1)$,
$\sigma\mto \sigma_{\psi_P}$
and $\beta_P\!: C^r(X,E_2)\to C^r(P,F_2)$,
$\sigma\mto \sigma_{\phi_P}$
are continuous linear maps.
By Proposition~\ref{pushforw1},
$(f_{\phi_P,\psi_P})_*$
is a $C^k$-map,
with $d^j((f_{\phi_P,\psi_P})_*)
=(d^j_2f_{\phi_P,\psi_P})_*$,
for all $j\in |k]$,
identifying $C^r(P,F_1)^{2^j}$
with $C^r(P,F_1^{2^j})$.
In view of
Lemma~\ref{famgivessec} and Lemma~\ref{anyatlas},
Proposition~\ref{inittop}
shows that $C^r(X,f)$ is
a mapping of class~$C^k$.
The mappings $\alpha_P$ and $\beta_P$ being linear,
given $j\in |k]$,
the Chain Rule
yields
\begin{eqnarray*}
\beta_P\circ d^jC^r(X,f)
& = & d^j(\beta_P\circ C^r(X,f))
=d^j((f_{\phi_P,\psi_P})_*\circ \alpha_P)
=d^j((f_{\phi_P,\psi_P})_*) \circ \alpha_P^{2^j}\\
& = &(d_2^jf_{\phi_P,\psi_P})_*\circ \alpha_P^{2^j}
=((\delta^jf)_{\phi_P,\psi_P^{2^j}})_*\circ \alpha_P^{2^j}
=\beta_P\circ C^r(X,\delta^jf)
\end{eqnarray*}
for all $P\in\cP$,
considering $\alpha_P^{2^j}\!:
C^r(X,E_1)^{2^j}\to C^r(P,F_1)^{2^j}$
as a mapping into $C^r(P,F_1^{2^j})$
in the last three terms,
and abbreviating $\psi_P^{2^j}:=
\psi_P\oplus\cdots\oplus \psi_P$ ($2^j$-fold sum
of local trivializations).
Thus $d^jC^r(X,f)=C^r(X,\delta^jf)$.
\end{proof}
\begin{thm}\label{pushforwK}
Suppose that $\pi_1\!: E_1\to X$ and
$\pi_2\!: E_2\to X$ are $C^r$-vector bundles
over the same $\sigma$-compact,
finite-dimensional base~$X$,
with typical fibres $F_1$,
resp., $F_2$.
Let $k\in \N_0\cup\{\infty\}$,
and suppose that $f\!: U \to E_2$ is a bundle
map, defined on an $\R$-balanced open neighbourhood~$U$
of the zero-section in~$E_1$, such that
\begin{itemize}
\item[\n (a)]
$f(0_x)=0_x$ for all $x\in X$,
\item[\n (b)]
$f$ is partially $C^k$ in the vector variable, and
\item[\n (c)]
The functions $\delta^jf$ are of
class $C^r$, for every $j\in |k]$.
\end{itemize}
Then
\[
C^{\,r}_K(X,f)\!: \;C^{\,r}_K(X,U)\to C^{\,r}_K(X,E_2),\;\;\;
\sigma \mto f\circ \sigma
\]
is a mapping of class~$C^k$,
for every compact subset~$K$ of~$X$.
For all $j\in |k]$,
we have
\[
d^jC^{\,r}_K(X,f)=C^{\,r}_K(X,\delta^jf)\!:C^{\,r}_K(X,U\oplus E_1^{2^j-1})
\to C^{\,r}_K(X,E_2),
\]
identifying $C^{\,r}_K(X,U)\times C^{\,r}_K(X,E_1)^{2^j-1}$
with $C^{\,r}_K(X,U\oplus E_1^{2^j-1})$ in the natural way.
\end{thm}
\begin{proof}
Let $\tau\in C^{\,r}_K(X,U)$ be given;
we show that
$C^{\,r}_K(X,f)$
is of class $C^k$ on a neighbourhood of~$\tau$.
First, we note that
every $x\in X$ has an open neighbourhood $B_x$
such that local trivializations
$\psi_x\!:\pi_1^{-1}(B_x)\to B_x\times F_1$ and
$\phi_x\!:\pi_2^{-1}(B_x)\to B_x\times F_2$
exist.
Then $Q:=\psi_x(U\cap \pi_1^{-1}(B_x))$ is an open neighbourhood
of $(x,\tau_{\psi_x}(x))$ in $B_x\times F_1$. As $U\cap E_x$ is
$\R$-balanced, furthermore $C:=\{x\}\times ([-1,1]\tau_{\psi_x}(x))
\sub Q$. Now~$C$ being compact, we find an open neighbourhood
$A_x\sub B_x$ of~$x$ and an open $\R$-balanced zero-neighbourhood
$H_x$ in~$F_1$ such that $A_x\times I_x\sub Q$,
where $I_x:=([-1,1]\tau_{\psi_x}(x))+H_x$
is an $\R$-balanced, open zero-neighbourhood in~$F_1$
containing $\tau_{\psi_x}(x)$. After shrinking
$A_x$, we may assume that $A_x$
is relatively compact, $\wb{A_x}\sub B_x$,
and $\tau_{\psi_x}(\wb{A_x})\sub I_x$.
Let $(h_x)_{x\in X}$
be a $C^r$-partition of unity subordinate to
the open cover $\{A_x\!:x\in X\}$ of~$X$,
such that $\Supp(h_x)\sub A_x$.
By compactness of~$K$, we have $\sum_{x\in F_K}h_x|_K=1$
for some finite subset~$F_K$ of~$X$.
For $x\in F_K$, let $H_x\!:X\to\R$ be a
$C^r$-function such that $0\leq H_x\leq 1$,
$H_x|_{\sSup(h_x)}=1$,
and $K_x:=\Supp(H_x)\sub A_x$.
Then
\[
V:=\{\sigma\in C^{\,r}_K(X,E)\!: (\forall x\in F_K)\;\;
\sigma_{\psi_x}(K_x)\sub I_x\}
\]
is an $\R$-balanced, open zero-neighbourhood in $C^{\,r}_K(X,E_1)$
which is contained in $C^{\,r}_K(X,U)$
(cf.\ proof of Lemma~\ref{thingsopen}),
and $\tau\in V$.
The claim will follow if we can show
that $C^{\,r}_K(X,f)|_V$ is of class~$C^k$.
Note that $f\circ \sigma=\sum_{x\in F_K}h_x\cdot (f\circ (H_x\cdot\sigma))$,
for all $\sigma\in V$.
Thus,
we have
\[
C^{\,r}_K(X,f)|_V
=\sum_{x\in F_K}
j_x\circ (r_x)^{-1} \circ
m_{h_x|_{A_x}} \circ
(f_{\phi_x,\psi_x}|_{A_x\times I_x})_* \circ
(m_{H_x|_{A_x}}
\circ
\theta_x)|_V^{C^r_{K\cap K_x}(A_x,I_x)},
\]
where
$\theta_x\!: C^{\,r}_K(X,E_1)\to C^r(A_x,F_1)$,
$\sigma\mto \sigma_{\psi_x}|_{A_x}$
and the multiplication maps
$m_{H_x|_{A_x}}\!: C^r(A_x,F_1)\to
C^r(A_x,F_1)$,
$\gamma\mto (H_x|_{A_x})\cdot \gamma$
and
$m_{h_x|_{A_x}}\!: C^{\,r}_{K\cap K_x}(A_x,F_2)
\to C^{\,r}_{K\cap K_x}(A_x,F_2)$
are continuous linear
and thus smooth
(Definition~\ref{dfntops}, Corollary~\ref{contmlt}),
the mapping\linebreak
$r_x\!:C^{\,r}_{K\cap K_x}(X,E_2)\to
C^{\,r}_{K\cap K_x}(A_x,F_2)$, $\sigma\mto \sigma_{\phi_x}|_{A_x}$
is an isomorphism of topological vector
spaces by Lemma~\ref{restrcts}
and Lemma~\ref{anyatlas},
the inclusion map
$j_x\!: C^{\,r}_{K\cap K_x}(X,E_2)\to C^{\,r}_K(X,E_2)$
is a continuous linear map,
and where $(f_{\phi_x,\psi_x}|_{A_x\times I_x})_*\!:
C^{\,r}_{K\cap K_x}(A_x,I_x)\to C^{\,r}_{K\cap K_x}(A_x,F_2)$
is a $C^k$-map by Proposition~\ref{pushforw0},
for all $x\in F_K$.
Thus $C^{\,r}_K(X,f)|_V$
is a $C^k$-map, being a composition
of $C^k$-maps.
\end{proof}
\section{Pushforward of $\cD^r$-sections\/}\label{sec5}
In this section, we prove our main result,
an analogue of
Theorem~\ref{pushforwK} for $\cD^r$-sections.
Our arguments make essential
use of an alternative, more
explicit description of the topology
on~$\cD^r(X,E)$
(in the spirit of
\cite{Sch}, Chapter~III, \S1),
which we provide first.
\begin{center}
{\bf Explicit description
of the topology on {\boldmath $\cD^r(X,E)$}}
\end{center}
\begin{numba}\label{hierloos}
Let $\pi\!: E\to X$ be a $C^r$-vector bundle
with a locally convex $\K$-vector space~$F$ as
typical fiber,
over a $\sigma$-compact, finite-dimensional
base~$X$, of dimension~$d$.
Let~$\Gamma$ be a set of
continuous seminorms $q\!:F\to [0,\infty[$
on~$F$ defining the topology.
We assume that~$\Gamma$ is directed
in the sense that, for every $q_1,q_2\in \Gamma$,
there exists $q_3\in \Gamma$ such that
$q_1\leq q_3$ and $q_2\leq q_3$ pointwise.
$\cK(X)$ denotes the
set of all compact subsets of~$X$.
\end{numba}
\begin{numba}\label{dstpalph}
There exists an open cover~$\cP$ of~$X$
such that, for every $P\in \cP$,
there exists
a local trivialization
$\psi_P\!:\pi^{-1}(P)\to P\times F$.
\end{numba}
\begin{numba}\label{defnUnVn}
Using \cite{Lan}, Chapter~II, \S3, Theorem~3.3,
we find
a locally finite open cover~$(\wt{U}_n)_{n\in \N}$
of~$X$,
a sequence $(\wt{\kappa}_n)_{n\in \N}$
of charts $\wt{\kappa}_n\!: \wt{U}_n\to \wt{V}_n$
of~$X$, where $\wt{V}_n$ is an open subset of~$\R^d$,
and relatively compact
open subsets~$U_n\sub \tilde{U}_n$
such that $X=\bigcup_{n\in\N}U_n$;
we set $\kappa_n:=\wt{\kappa}_n|_{U_n}^{V_n}$,
where $V_n:=\wt{\kappa}_n(U_n)$ (which is a relatively compact, open subset
of~$\wt{V}_n$).
Since each of the relatively compact sets~$U_n$
may be replaced by finitely
many open subsets~$W$ of~$U_n$ such that
$\wb{W}\sub P$ for some $P\in\cP$,
and then $\wt{W}\sub P$ for some open neighbourhood~$\wt{W}$
of $\wb{W}$ in~$\wt{U}_n$,
it is clear that we may assume that
the open cover $(\wt{U}_n)_{n\in \N}$ is subordinate
to~$\cP$. Thus, given $n\in\N$, we have
$\wt{U}_n\sub P_n$ for some $P_n\in\cP$.
We set $\psi_n:=\psi_{P_n}|_{\pi^{-1}(\wt{U}_n)}\!:
\pi^{-1}(\wt{U}_n)\to\wt{U}_n\times F$.
\end{numba}
\begin{numba}\label{unsch}
Given an open subset $Y$ of $\R^d$, $k\in |r]$,
$q\in\Gamma$, and compact subset $K$ of $Y$,
we define
\[
\|\gamma\|_{K,q,k}:= \max_{\alpha\in\N_0^d, |\alpha|\leq k}
\sup_{x\in K} q\left(
\partial^\alpha \gamma\, (x)\right)
\]
for $\gamma \in C^r(Y,F)$,
using standard multi-index notation
for the partial derivatives.
Then $(\|.\|_{K,q,k})_{K\in\cK(Y),q\in\Gamma,k\in|r]}$
is a family of continuous seminorms
on $C^r(Y,F)$ determining the given
locally convex topology (cf.\ \cite{Glo}, Proposition~4.4).
\end{numba}
\begin{numba}\label{dstpomeg}
Given $n\in \N$, $q\in \Gamma$,
and $k\in |r]$,
we define
\[
\|\sigma\|_{n,q,k}:= \| \sigma_{\psi_n}\circ
\wt{\kappa}_n^{-1}\|_{\wb{V_n},\, q,k}
\]
for $\sigma\in C^r(X, E)$.
The linear mappings $C^r(X,E)
\to C^r(\wt{V}_n,F)$, $\sigma\mto \sigma_{\psi_n}\circ
\wt{\kappa}^{-1}_n$ being continuous
by definition of the topology on $C^r(X,E)$,
we deduce from {\bf \ref{unsch}}
that $\|.\|_{n,q,k}$ is a continuous seminorm on $C^r(X,E)$,
for any $n\in \N$, $q\in\Gamma$, and $k\in |r]$.
\end{numba}
Then we have:
\begin{la}\label{expldestop}
The topology on $C^r(X,E)$
described in Definition~{\n \ref{dfntops}}
coincides with the locally convex vector
topology determined by the family of
seminorms $(\|.\|_{n, q, k})$,
where $n\in\N$, $q\in \Gamma$, and $k\in |r]$.
\end{la}
\begin{proof}
We already observed that each seminorm
$\|.\|_{n,q,k}$ is continuous;
hence the locally convex vector topology~$\cO$
on $C^r(X,E)$ determined by the family
$(\|.\|_{n,q,k})_{n,q,k}$ is coarser than the given topology.
On the other hand, we know from Lemma~\ref{anyatlas} that
the given topology on $C^r(X,E)$
is initial with respect to
the sequence $(\theta_n)_{n\in \N}$
of linear maps
\[
\theta_n\!: C^r(X,E)\to C^r(V_n,F),\;\;\;
\sigma\mto \sigma_{\psi_n}|_{V_n}\circ \kappa_n^{-1}.
\]
The topology on $C^r(V_n,F)$
is determined by the seminorms
$\|.\|_{K,q,k}$, where $K\in \cK(V_n)$, $q\in \Gamma$,
$k\in |r]$ (see {\bf \ref{unsch}}).
As $\|.\|_{K,q,k}\circ \theta_n\leq \|.\|_{n,q,k}$
pointwise on $C^r(X,E)$,
we easily deduce that the given topology on $C^r(X,E)$
is coarser than~$\cO$. Thus both topologies coincide.
\end{proof}
\begin{defn}\label{defineVs}
Given any sequences $q=(q_n)_{n\in \N}$,
$k=(k_n)_{n\in \N}$,
and $e=(\ve_n)_{n\in\N}$
of seminorms $q_n\in\Gamma$,
natural numbers $k_n\in |r]$,
and positive real numbers $\ve_n>0$, 
we define
\[
\cV(q, k , e):=\{
\sigma \in \cD^r(X,E)\!:\; \mbox{$\|\sigma\|_{n, q_n, k_n}<\ve_n$ for
all $n\in \N$}\,\}.\]
\end{defn}
\begin{prop}\label{expldestopII}
For $(q, k, e)\in \Gamma^\N\times |r]^\N\times
(\R^+)^\N$, the sets
$\cV(q, k , e )$
form a basis for the filter of zero-neighbourhoods
in~$\cD^r(X,E)$. Every $\cV(q, k , e)$
is open in~$\cD^r(X,E)$.
\end{prop}
\begin{proof}
Using that $(\wt{U}_n)_{n\in \N}$ is a locally finite
open cover of~$X$,
we easily
see that the sets $\cV( q , k , e )$
form a basis of convex, open zero-neighbourhoods
for a locally convex Hausdorff vector topology~$\cO$
on $\cD^r(X,E)$, which makes the inclusion maps
$C^{\, r}_K(X,E)\to\cD^r(X,E)$ embeddings
of topological vector spaces in view of Lemma~\ref{expldestop}.
We show that $(\cD^r(X,E),\cO)$
is the locally convex direct limit
of its subspaces $C^{\,r}_K(X,E)$,
by verifying the universal property.
Thus, suppose that $(\lambda_K)_{K\in\cK(X)}$
is a family of continuous
linear maps $\lambda_K\!: C^{\, r}_K(X,E)\to Y$
into a locally convex space~$Y$ such that
$\lambda_C|_{C^{\,r}_K(X,E)}=\lambda_K$ for all $K,C\in \cK(X)$
such that~$K\sub C$.
As $\cD^r(X,E)=\dl\, C^{\, r}_K(X,E)$
in the category of vector spaces,
there is a unique linear map $\lambda\!: \cD^r(X,E)\to Y$
such that $\lambda|_{C^{\, r}_K(X,E)}=\lambda_K$ for all~$K$.
It only remains to show that~$\lambda$
is continuous. To this end, let a continuous
seminorm~$p$ on~$Y$ and $\ve>0$ be given.
We let $(h_n)_{n\in \N}$
be a $C^r$-partition of unity
subordinate to the open covering $(U_n)_{n\in \N}$,
(\cite{Lan}, Chapter~II, \S3, Corollary~3.4).
Thus $0\leq h_n\leq 1$, $C_n:=\Supp(h_n)\sub U_n$,
and $\sum_{n=1}^\infty h_n = 1$.
The topology on
$C^{\, r}_{C_n}(X,E)$ is determined by the directed family of
seminorms $(\|.\|_{n,q,k})_{q\in\Gamma, k\in |r]}$
(restricted to that space),
as a consequence of Lemma~\ref{restrcts}
and Lemma~\ref{anyatlas}.
Now $\lambda_{C_n}$ being continuous,
we deduce that there exists $q_n\in \Gamma$, $k_n\in |r]$ and
$\delta_n>0$ such that $p(\lambda_{C_n}(\sigma))<\frac{\ve}{2^n}$
for all $\sigma\in C^{\, r}_{C_n}(X,E)$ such that
$\|\sigma\|_{n,q_n,k_n}<\delta_n$.
Using the Leibniz Rule,
we find $\ve_n>0$
such that $\|(h_n\circ \wt{\kappa}_n^{-1})\cdot
\gamma\|_{\wb{V_n},q_n,k_n}<\delta_n$
for all $\gamma\in C^r(\wt{V}_n,F)$
such that $\|\gamma\|_{\wb{V_n},q_n,k_n}<\ve_n$
(see \cite{Glo}, proof of Proposition~4.8).
Thus $\|h_n\cdot \sigma\|_{n,q_n,k_n}<\delta_n$
for all $\sigma\in C^r(X,E)$ such that $\|\sigma\|_{n,q_n,k_n}<\ve_n$.
We set $q:=(q_n)_{n\in\N}$,
$k:=(k_n)_{n\in \N}$,
and $e:=(\ve_n)_{n\in \N}$.
Then $p(\lambda(\sigma))<\ve$ for all
$\sigma\in \cV(q,k,e)$.
In fact, we can find $N\in \N$ such that $U_n\cap \Supp(\sigma)=\emptyset$
for all $n>N$. Then
$\lambda(\sigma)=\sum_{n=1}^N \lambda_{C_n}(h_n\cdot \sigma)$,
and thus $p(\lambda(\sigma))\leq \sum_{n=1}^N \frac{\ve}{2^n}<\ve$.
\end{proof}
\begin{center}
{\bf The Main Result}
\end{center}
\begin{thm}\label{pushforw3}
Suppose that $\pi_1\!: E_1\to X$ and
$\pi_2\!: E_2\to X$ are $C^r$-vector bundles
over the same $\sigma$-compact, finite-dimensional base~$X$,
with typical fibres $F_1$,
resp., $F_2$.
Let $k\in \N_0\cup\{\infty\}$, $U$ be an $\R$-balanced open
neighbourhood of $0_X$ in~$E_1$,
and suppose that $f\!: U\to E_2$ is a bundle
map such that
\begin{itemize}
\item[\n (a)]
$f(0_x)=0_x$ for all $x\in X$,
\item[\n (b)]
$f$ is partially $C^k$ in the vector variable, and
\item[\n (c)]
The functions $\delta^jf$ are of
class $C^r$, for every $j\in |k]$.
\end{itemize}
Then
\[
\cD^r(X,f)\!: \cD^r(X,U)\to \cD^r(X,E_2),\;\;\;
\sigma \mto f\circ \sigma
\]
is a mapping of class~$C^k$.
\end{thm}
\begin{proof}
Note that if $k>0$,
then for $\sigma\in \cD^r(X,U)$ and $\tau\in \cD^r(X,E_1)$,
we have
\[
h^{-1}(f\circ (\sigma +h \tau)-f\circ
\sigma) \in C^{\,r}_K(X,U)
\]
for all $0\not= h\in \R$ sufficiently small,
where $K:=\Supp(\sigma)\cup\Supp(\tau)$.
Since $\cD^r(X,E_1)$
induces the usual
topology on $C^{\,r}_K(X,E_1)$,
we deduce from Theorem~\ref{pushforwK}
that $d(\cD^r(X,f))\,(\sigma,\tau)=d(C^{\,r}_K(X,f))\,(\sigma,\tau)$
exists.
Repeating this argument, we find
that
$d^j\cD^r(X,f)$ exists for all $j\in |k]$,
and is given by $d^j\cD^r(X,f)=\cD^r(X,\delta^j f)$,
where we identify $\cD^r(X,U)\times
\cD^r(X,E_1)^{2^j-1}$ with $\cD^r(X,U\oplus E_1^{2^j-1})$.
As the mappings $\cD^r(X,\delta^j f)$ are of the form described in the
lemma,
it suffices to show that $\cD^r(X,f)$ is a
continuous map, for each~$f$ as above.
In other words, it suffices to prove the theorem
when $k=0$, which we assume now.\vspace{3 mm}\\
For $i\in \{1,2\}$,
we let $\Gamma_i$ be a directed
set of continuous seminorms
defining the locally convex vector topology on~$F_i$,
and define seminorms $\|.\|_{n,q,k}$ on $C^r(X,E_i)$
for $n\in \N$, $q\in \Gamma_i$, $k\in |r]$
as well as an open basis of zero-neighbourhoods
$\cV_i(q,k,e)$ for $\cD^r(X,E_i)$
(where $q\in \Gamma_i^\N$,
$k \in |r]^\N$, $e\in (\R^+)^\N$)
as described in {\bf \ref{dstpalph}}\,--\,{\bf \ref{defineVs}},
where in {\bf \ref{dstpalph}},
we assume now that local trivializations
$\psi^i_P\!: \pi^{-1}_i(P)\to P\times F_i$
exist for $i\in \{1,2\}$, and choose
the same sequence $(\wt{U}_n)$ for both
$i=1$ and $i=2$, and the same sequence~$(U_n)$.

Fix $\tau\in \cD^r(X,U)$;
our goal is to show that $\cD^r(X,f)$ is continuous
at $\tau$. To this end, let
$p=(p_n)\in \Gamma_2^\N$,
$\ell=(\ell_n)\in |r]^\N$, and $e=(\ve_n)\in (\R^+)^\N$
be given; we have to find a neighbourhood~$W$
of~$\tau$ in $\cD^r(X,U)$ such that
$\|f\circ \tau -f\circ \sigma\|_{n,p_n,\ell_n}<\ve_n$
for all $\sigma\in W$ and $n\in \N$.
We choose
$C^r$-functions $h_n\!: X\to \R$
such that $\im(h_n)\sub [0,1]$,
$\Supp(h_n)\sub \wt{U}_n$,
and such that $h_n$ is constantly~$1$
on some neighbourhood of~$\wb{U}_{\!n}$.
This is possible as~$\wb{U}_{\!n}$
is a compact subset of~$\wt{U}_n$.
For each~$n$, define
\[
M_n:=\{m\in \N\!:
\Supp(h_n)\cap U_m\not=\emptyset\}.
\]
Then $M_n$ is a finite set,
and $\Supp(h_n)\sub \bigcup_{m\in M_n} U_m$.
Now $C^{\,r}_{\sSup(h_n)}(X,f)\!: C^{\, r}_{\sSup(h_n)}(X,U)\to
C^{\, r}_{\sSup(h_n)}(X,E_2)$
being continuous
by Theorem~\ref{pushforwK},
we easily deduce from Lemma~\ref{expldestop}
that there exists
$s_n\in \Gamma_1$, $j_n\in |r]$,
and $\beta_n>0$ such that
\begin{equation}\label{tk}
\|f\circ \sigma-f\circ (h_n\cdot \tau)\|_{n,p_n,\ell_n}<\ve_n
\end{equation}
for all $\sigma \in C^{\,r}_{\sSup(h_n)}(X,U)$
such that $\|\sigma- h_n\cdot \tau\|_{m,s_n,j_n}<\beta_n$
for all $m\in M_n$.\footnote{Note that $\im(h_n\cdot \tau)\sub U$
here as $\im(h_n)\sub [0,1]$ and $U$ is $\R$-balanced.
Also note that $\|\sigma\|_{m,s,j}=0$
if\linebreak
\hspace*{5.5 mm} $\sigma\in C^{\,r}_{\sSup(h_n)}(X,E_1)$
and $m\in \N\,\take\, M_n$, for all
$s\in \Gamma_1$ and
$j\in |r]$.}
Now, using the Leibniz Rule as in the proof
of Proposition~\ref{expldestopII}, we find real numbers
$\rho_n>0$ such that, for all $m\in M_n$, we have
$\|h_n\cdot \sigma\|_{m,s_n,j_n}<\beta_n$
for all $\sigma\in C^r(X,E)$
such that $\|\sigma\|_{m,s_n,j_n}<\rho_n$.
Note that since $(\wt{U}_n)_{n\in\N}$
is a locally finite covering of~$X$,
and $\Supp(h_n)\sub \wt{U}_n$,
the sets $N_m:=\{n\in \N\!: m\in M_n\}$
are finite, for every $m\in \N$ (N.B. $\wb{U}_{\!m}$ is compact).
We set $\delta_m:=\min_{n\in N_m}\rho_n$,
$k_m:=\max_{n\in N_m}j_n$,
and choose a seminorm $q_m\in \Gamma_1$
such that $q_m\geq s_n$ for all
$n\in N_m$, which is possible as~$\Gamma_1$ is
directed. We set $q:=(q_n)_{n\in \N}$,
$k:=(k_n)_{n\in \N}$, and $\delta:=(\delta_n)_{n\in \N}$.
Then, given any
$\sigma\in W:=(\tau+\cV_1(q,k,\delta))\cap \cD^r(X,U)$,
for all $n\in \N$ we have
\[
\|f\circ \sigma-f\circ \tau\|_{n,p_n,\ell_n}=
\|f\circ (h_n\cdot \sigma)-f\circ (h_n\cdot \tau)\|_{n,p_n,\ell_n}<\ve_n
\]
by Equation\,(\ref{tk}),
since $\|h_n\cdot \sigma- h_n\cdot \tau\|_{m, s_n,j_n}
=\|h_n\cdot (\sigma-\tau)\|_{m,s_n,j_n}<\beta_n$
as $\|\sigma-\tau\|_{m,s_n,j_n}\leq
\|\sigma-\tau\|_{m,q_m,k_m}<\delta_m\leq \rho_n$
for each $m\in M_n$.
Thus $\cD^r(X,f)$ is continuous at~$\tau$.
\end{proof}
An analogue of Theorem~\ref{pushforw3}
in the setting of convenient differential
calculus, for fibre preserving mappings
of class~$c^\infty$ between smooth bundles,
was already established in~\cite{KaM},
Corollary~30.10.
The main novel features of
Theorem~\ref{pushforw3} are,
first and foremost, that it ensures
continuous differentiability of pushforwards
in a much stronger sense than being a
$c^\infty$-map in the sense of
convenient differential calculus (when $k=\infty$).
In particular, it shows that pushforwards
are continuous.
Secondly, the theorem extends beyond
the case of smooth sections,
and addresses pushforwards of $C^r$-sections,
at no extra cost. The techniques applied
in the proof generalize from pushforwards to so-called ``almost local''
mappings between open subsets of spaces of compactly
supported sections; such mappings are encountered
in connection with diffeomorphism groups
of non-compact manifolds~\cite{LOC}.
\section{Analytic mappings between spaces of sections}\label{sec6}
In this section, we describe conditions ensuring
analyticity of pushforwards.
\begin{prop}\label{propcxana}
Let $\pi_j\!:E_j\to X$ be $C^r$-vector
bundles,
whose typical fibres $F_j$ are
complex locally convex spaces,
over the same finite-dimensional, $\sigma$-compact base~$X$,
for $j\in\{1,2\}$.
Suppose that $f\!:U\to E_2$
is a bundle map, defined on an open $\R$-balanced
neighbourhood~$U$ of~$0_X$ in~$E_1$,
such that
\begin{itemize}
\item[\n (a)]
$f$ is partially $C^\infty$ in the vector variable;
\item[\n (b)]
$\delta^jf$ is of class $C^r$, for all $j\in \N_0$;
\item[\n (c)]
For every local trivialization $\psi$ of $E_1$
and $\phi$ of $E_2$,
the mapping
\[
d_2f_{\phi,\psi}= (\delta f)_{\phi,\psi\oplus\psi}\!:
U_{\phi,\psi}\times F_1\to F_2
\]
$($see Definition~{\n \ref{defnparvect})} is complex linear in~$F_1$.
\end{itemize}
Then $\cD^r(X,f)\!:\cD^r(X,U)\to \cD^r(X,E_2)$
is complex analytic, and so is the mapping
$C^{\,r}_K(X,f)\!: C^{\,r}_K(X,U)\to C^{\,r}_K(X,E_2)$,
for every compact subset~$K$ of~$X$.
\end{prop}
\begin{proof}
By Theorem~\ref{pushforwK}
and Theorem~\ref{pushforw3},
the mappings $C^{\,r}_K(X,f)$ and $\cD^r(X,f)$ are
smooth.
For each $\sigma\in C^{\,r}_K(X,U)$,
the differential
$C^{\,r}_K(X,f)(\sigma, \sbull)\!: C^{\,r}_K(X,E_1)\to C^{\,r}_K(X,E_2)$
is complex linear,
as a consequence of Hypothesis\,(c).
Similarly, $d\cD^r(X,f)(\sigma,\sbull)$ is complex linear
for each $\sigma\in \cD^r(X,U)$.
Now \cite{RES}, Lemma~2.5 shows that $C^{\,r}_K(X,f)$
and $\cD^r(X,f)$ are complex analytic mappings.
\end{proof}
As regards real analytic mappings, we leave
the general framework of spaces of sections
and content
ourselves with a result establishing
real analyticity for suitable mappings
between function spaces.
The following proposition,
designed for application in Section~\ref{secmapgps},
generalizes \cite{Glo},
Corollary~4.17.\vspace{3 mm}\\
Given a $\sigma$-compact, finite-dimensional manifold
$X$, locally convex space~$E$,
and open $\R$-balanced
neighbourhood $U$ of the zero-section
$X\times \{0\}$
in the trivial bundle $X\times E$,
we define $\cD^r(X,U):=\{\gamma\in \cD^r(X,E)\!:
\,\im(\id_X\times \gamma)\sub U\}$.
As a consequence of
Lemma~\ref{anyatlas} and Lemma~\ref{thingsopen},
$\cD^r(X,U)$ is an open zero-neighbourhood
in $\cD^r(X,E)$.
\begin{prop}\label{proprealana}
Suppose that $E$ and $F$ are real locally convex spaces,
$X$ is a $\sigma$-compact finite-dimensional $C^r$-manifold,
$U$ an open $\R$-balanced neighbourhood of
$X\times\{0\}$ in $X\times E$,
$P$ a real analytic manifold modelled
over some locally convex space~$Z$,
$\wt{U}$ an open subset of $P\times E$,
and $\gamma\!: X\to P$ a $C^r$-mapping
such that $(\gamma\times \id_E)(U)\sub \wt{U}$.
Let $\tilde{f}\!: \wt{U}\to F$ be a real analytic
mapping.
We define $f:=\tilde{f}\circ (\gamma\times \id_E)|_U^{\wt{U}}$,
and assume that $f(x,0)=0$
for all $x\in X$.
Then
\[
f_*\!: \cD^r(X,U)\to \cD^r(X,F),\;\;\;
\gamma\mto f\circ (\id_X,\gamma)
\]
is real analytic on $\cD^r(X,Q)$
for some open, $\R$-balanced neighbourhood~$Q$
of $X\times \{0\}$ in~$U$.
\end{prop}
\begin{proof}
As~$P$ is a real analytic manifold,
for every $a\in P$
we find a diffeomorphism
$\phi_a\!: W_a\to P_a$ of real analytic manifolds
from an open zero-neighbourhood~$W_a$ in~$Z$
onto an open neighbourhood~$P_a$ of~$a$ in~$P$
such that~$\phi_a(0)=a$.
Assume $a\in \im(\gamma)$ now.
After shrinking~$P_a$ and~$W_a$, we may assume
that $P_a\times B_a\sub \wt{U}$
for some open $0$-neighbourhood $B_a$ in~$E$,
since~$\wt{U}$ is an open neighbourhood of $\im(\gamma)\times \{0\}$,
Then $\theta_a\!:W_a\times B_a\to F$,
$\theta_a(w,x):=
\tilde{f}(\phi_a(w),x)$
is a real analytic mapping and hence extends to a
complex analytic mapping $\tilde{\theta}_a\!:A_a\to F_\C$,
defined on an open neighbourhood~$A_a$ of
$W_a\times B_a$ in $Z_\C\times E_\C$.
Shrinking~$W_a$ if necessary, we may assume
that~$A_a$ contains a $0$-neighbourhood
of the form $H_a\times (R_a+i\, S_a)$
for an open neighbourhood~$H_a$ of~$W_a$ in~$Z_\C$,
an open, $\R$-balanced zero-neighbourhood~$R_a\sub B_a$ in~$E$,
and an open, symmetric, convex
zero-neighbourhood $S_a$
in~$E$.
If $a$, $b\in \im(\gamma)$, then for every $p\in P_a\cap P_b$,
the prescriptions
$x\mto \tilde{\theta}_a(\phi_a^{-1}(p),x)$
and $x\mto \tilde{\theta}_{b}(\phi_{b}^{-1}(p),x)$
define complex analytic mappings
\[
(R_a+iS_a)\cap
(R_b+iS_b)=(R_a\cap R_b)+i(S_a\cap S_b)
\to F_\C
\]
which coincide
on $R_a\cap R_b$ (where they coincide with
$\tilde{f}(p, \sbull)|_{R_a\cap R_b}$),
and which therefore coincide. We abbreviate
\[
\wt{V}:=\bigcup_{a\in \smim(\gamma)} P_a\times (R_a+iS_a),
\]
which is an open neighbourhood of $\im(\gamma)\times\{0\}$
in $P\times E_\C$.
By the preceding,
\[
\check{f}\!: \wt{V} \to F_\C,\;\;\;\;
\check{f}(p,x):=\tilde{\theta}_a(\phi_a^{-1}(p),x)\;\;\;
\mbox{if $(p,x)\in P_a\times(R_a+iS_a)$, where $a\in \im(\gamma)$}\]
is a well-defined smooth mapping
such that $\check{f}(p,\sbull)\!: \bigcup_{a\in \smim(\gamma)\;
s.t.\; p\in P_a}(R_a+iS_a) \to F_\C$
is complex analytic for each $p\in \bigcup_{a\in\smim(\gamma)}P_a$.
Set $V:=(\gamma\times \id_{E_\C})^{-1}(\wt{V})\sub X\times E_\C$,
and $Q:=V\cap (X\times E)$.
Clearly
$h:=\check{f}\circ (\gamma\times \id_{E_\C})|_V^{\wt{V}}$,
$V\to F_\C$
is partially $C^\infty$ in the second argument.
For any $(x,y)\in V$,
there are open neighbourhoods
$\wt{M}\sub P$ of $\gamma(x)$
and $N\sub E_\C$ of~$y$ such that $\wt{M}\times N\sub \wt{V}$;
set $M:=\gamma^{-1}(\wt{M})$.
Then $M\times N\sub V$, and
$d_2^n(h|_{M\times N})=d_2^n(\check{f}|_{\wt{M}\times N})
\circ (\gamma|_M^{\wt{M}}\times \id_{T^nN})\!:
M\times T^n N\to E_\C$ is of class~$C^r$ for each $n\in \N_0$.
Thus $d_2^nh$ is of class $C^r$ for each $n\in \N_0$.
Furthermore, clearly
$h(x,\sbull)$ is complex analytic for
each $x\in X$. Finally, $h(x,0)=f(x,0)=0$ for
all $x\in X$,
noting that~$h$ extends $f|_Q$
by construction.
By Proposition~\ref{propcxana},
the mapping
$h_*\!: \cD^r(X, V)\to \cD^r(X,F_\C)=\cD^r(X,F)_\C$
is complex analytic.
Thus $f_*|_{\cD^r(X,Q)}=h_*|_{\cD^r(X,Q)}$
is real analytic.
\end{proof}
\section{Spaces of sections as topological modules}\label{sec7}
Spaces of sections in vector bundles
can be considered as modules over various function
algebras.
In this section, we investigate continuity
properties of the module operations.\vspace{3 mm}\\
The following algebras are of interest.
\begin{prop}\label{topALG}
Suppose that $A$ is an associative locally convex topological algebra.
If $r\in \N\cup\{\infty\}$,
let $X$ be a $($not necessarily finite-dimensional$)$
$C^r$-manifold. If $r=0$, let~$X$ be any Hausdorff
topological space. Then $C^r(X,A)$ is a
locally convex topological algebra
with respect to pointwise operations.
Furthermore,
if~$X$ is a $\sigma$-compact finite-dimensional
manifold, then $\cD^r(X,A)$
is a locally convex topological algebra.
\end{prop}
\begin{proof}
(See also \cite{GOO}, Proposition~2.1
and Remark 2.2).
Identifying $C^r(X,A)\times C^r(X,A)$
with $C^r(X, A\times A)$,
the algebra multiplication on $C^r(X,A)$
is given by
\[
C^r(X,\mu)\!: C^r(X,A\times A)\to C^r(X,A),\;\;\;
\gamma\mto \mu\circ \gamma,
\]
where the algebra multiplication $\mu\!:
A\times A$ is a smooth map. By Corollary~\ref{functorial1},
$C^r(X,\mu)$ is smooth and thus continuous.
The proof for $\cD^r(X,A)$ follows the same lines,
using \cite{Glo}, Corollary~4.16
(or Theorem~\ref{pushforw3} above)
instead of Corollary~\ref{functorial1}.
\end{proof}
\begin{thm}\label{modopcont}
Let $\pi\!: E\to X$ be a $C^r$-vector bundle,
whose fibre is a locally convex $\K$-vector
space~$F$.
Then
$C^r(X,E)$ is a topological $C^r(X,\K)$-module,
and, provided $X$ is finite-dimensional and $\sigma$-compact,
$\cD^r(X,E)$ is a topological $\cD^r(X,\K)$-module.
\end{thm}
\begin{proof}
The function space
$\cD^r(X,\K)$ can be identified with
$\cD^r(X,X \times \K)$,
the space of $\cD^r$-sections
in the trivial bundle
$\pr_1\!: X\times \K\to X$
with fibre~$\K$ (cf.\ Lemma~\ref{anyatlas}).
Thus $\cD^r(X,\K)\times \cD^r(X,E)\isom \cD^r(X, (X\times \K)\oplus
E)$, and using this identification,
the multiplication map $\cD^r(X,\K)\times \cD^r(X,E)\to
\cD^r(X,E)$ has the form
\[
\cD^r(X,\mu)\!: \cD^r(X, (X\times \K)\oplus E)\to \cD^r(X,E),\]
where $\mu\!: (X\times \K)\oplus E \to E$
is defined $\mu((x,z),v):=zv\in E_x$ (scalar multiplication)
for all $x\in X$, $z\in \K$, and $v\in E_x$.
Given any local trivialization $\psi\!:\pi^{-1}(X_\psi)\to X_\psi\times
F$ of~$E$,
using the global trivialization $\phi:=\id\!:X\times \K\to
X\times \K$
we have $\mu_{\psi, \phi\oplus \psi}(x,z,v)=
zv\in F$, for all $(x,z,v)\in X_\psi\times\K\times F$,
showing that $\mu_{\psi, \phi \oplus \psi}\!:
X_\psi\times\K\times F\to F$ is a mapping of class~$C^r$.
Thus $\mu$ is a $C^r$-bundle map.
By Theorem~\ref{pushforw3} (applied with $k=0$),
$\cD^r(X,\mu)$ is continuous.
The assertion concerning $C^r(X,E)$ can be proved analogously,
using Theorem~\ref{pushforw2}
instead of Theorem~\ref{pushforw3}.
\end{proof}
\begin{rem}
If $A$ is a locally convex topological algebra
and $X$ a $C^r$-manifold,
we define a locally trivial bundle of locally convex topological
$A$-modules
as a $C^r$-vector bundle
$\pi\!: E\to X$
whose typical fibre~$F$
is a locally convex topological $A$-module,
and equipped with an atlas $\cA$ of local trivializations
such that $\im(g_{\phi,\psi})$
consists of topological $A$-module
automorphisms of~$F$, for all $\phi,\psi\in\cA$.
In this case,
we find along the lines of the preceding proof
that
$C^r(X,E)$ is a topological $C^r(X,A)$-module,
and, provided the base~$X$ is finite-dimensional
and $\sigma$-compact, $\cD^r(X,E)$ is
a topological $\cD^r(X,A)$-module (under pointwise
operations).
\end{rem}
Recall that a bilinear mapping
$\beta\!: E\times F\to G$ between locally convex
spaces is {\em hypocontinuous\/}
if and only if $\beta|_{A\times F}\!:
A\times F\to G$ and $\beta|_{E\times B}\!:
E\times B\to G$ are continuous mappings,
for any bounded subsets $A\sub E$
and $B\sub F$. Every hypocontinuous
bilinear map is separately continuous
and sequentially continuous;
it need not be continuous.
It is easy to see that if $\beta\!:E\times F\to G$
is a bilinear map
such that
$\beta|_{A\times F}$ and $\beta|_{E\times B}$
are continuous at~$(0,0)$
for all bounded subsets $A\sub E$, $B\sub F$ containing the origin, then
$\beta$ is hypocontinuous.\vspace{3 mm}\\
Let $E\to X$ be a $C^r$-vector bundle
over a $\sigma$-compact, finite-dimensional
base~$X$, having the locally convex $\K$-vector
space $F$ as its typical fibre.
Then
pointwise multiplication makes $\cD^r(X,E)$
a $C^r(X,\K)$-module. The following can be said:
\begin{prop}\label{onlyifcomp}
The module multiplication
\[
\mu\!: C^r(X,\K)\times \cD^r(X,E)\to \cD^r(X,E)
\]
is hypocontinuous.
If $F\not=\{0\}$, then $\mu$ is continuous
if and only if~$X$ is compact.
\end{prop}
\begin{proof}
Let $B\sub \cD^r(X,E)$
be a bounded subset.
Then $B\sub C^{\,r}_K(X,E)$
for some compact
subset $K\sub X$ (cf.\ \cite{Sae}, Chapter~II,
assertion~6.5).
Choose a function $h\in C^r(X,\R)$
such that $h|_K=1$ and
$L:=\Supp(h)$ is compact.
Then we have, using the continuous
inclusion maps $i\!: B\to C^{\,r}_L(X,E)$
and $j\!: C^{\,r}_L(X,E)\to \cD^r(X,E)$,
\[
\mu|_{C^r(X,E)\times B} = j\circ \nu\circ (m_h \times i),\]
where the module multiplication
$\nu\!: C^{\,r}_L(X,\K)\times C^r(X,E)\to C^{\,r}_L(X,E)$
is continuous by Theorem~\ref{modopcont},
and so is multiplication
$m_h\!: C^r(X,E)\to C^{\,r}_L(X,E)$,
$\gamma\mto h\cdot \gamma$
(Corollary~\ref{contmlt}).
Thus
$\mu|_{C^r(X,\K)\times B}$ is continuous.
To establish the hypocontinuity of~$\mu$,
it only remains to show
that
$\mu|_{A\times \cD^r(X,E)}$
is continuous at $(0,0)$,
for every bounded subset $A\sub C^r(X,\K)$
containing the origin.
We shall make use of the constructions and notation
described in {\bf \ref{hierloos}}\,--\,{\bf \ref{defineVs}}.
Suppose that a zero-neighbourhood
$\cV(q,k,e)$ in $\cD^r(X,E)$
is given, where $q=(q_n)\in \Gamma^\N$, $k=(k_n)\in |r]^\N$,
and $e=(\ve_n)\in (\R^+)^\N$.
As $A$ is bounded,
so is $H_n:=\{f\circ \wt{\kappa}_n^{-1}\!: f\in A\}\sub
C^r(\wt{V}_n,\K)$,
for each $n\in \N$.
Thus $\sup_{f\in H_n}\|f\|_{\wb{V_n},q_n,k_n}<\infty$.
Using the Leibniz Rule,
we easily find $\delta_n>0$ such that
$\|f\cdot \gamma\|_{\wb{V_n},q_n,k_n}<\ve_n$
for all $f\in H_n$
and all $\gamma\in C^r(\wt{U}_n,F)$
such that $\|\gamma\|_{\wb{V_n},q_n,k_n}<\delta_n$.
Then $\mu(A\times \cV(k,q,(\delta_n)))\sub \cV(k,q,e)$.
We deduce that $\mu|_{A\times \cD^r(X,E)}$
is continuous at $(0,0)$.

If $X$ is compact, then $\cD^r(X,E)=C^r(X,E)$,
and we are in the situation of Theorem~\ref{modopcont}.
Thus, to prove the final
assertion, assume that $X$ is non-compact
and $F\not=\{0\}$.
We let $W:=\cV(q,e,k)$ be a zero-neighbourhood
in $\cD^r(X,E)$ as described in Definition~\ref{defineVs},
where $q\in \Gamma^\N$
is chosen such that $q_n\not = 0$ for all
$n\in \N$, $e\in (\R^+)^\N$,
and $k=(k_n)_{n\in \N}$ with
$k_n:=0$ for all $n\in \N$.
If $\mu$ was continuous, we could find
zero-neighbourhoods $P$ in $C^r(X,\K)$ and
$Q$ in $\cD^r(X,E)$ such that $\mu(P\times Q)\sub W$.
In view of the definition of the topology on
$C^r(X,\K)$,
we can find a compact set $K\sub X$ such
that $\{f\in C^r(X,\K)\!: f|_K=0\,\}\sub P$.
As $X$ is non-compact, we find $n\in \N$
such that $\emptyset\not=U_n\sub X\,\take\, K$,
where $U_n$ is as in {\bf \ref{defnUnVn}}.
Pick $x_0\in U_n$. Making use of the local trivialization $\psi_n$
(given by {\bf \ref{defnUnVn}}),
is easy to construct $\sigma\in \cD^r(X,E)$
such that $s:=q_n(\sigma_{\psi_n}(x_0))\not=0$;
after replacing~$\sigma$ by a small non-zero multiple,
we may assume that $\sigma\in Q$.
We also easily find $h\in C^r(X,\K)$ such that
$\Supp(h)\sub U_n$
and $h_n(x_0)=1$.
Then $th\in P$ for all $t\in \R$
and thus $th\sigma\in W$.
Here $q_n((th\sigma)_{\psi_n}(x_0))=|t|s$
which can be made arbitrarily large.
Therefore $th\sigma\not\in W$ for sufficiently large $|t|$,
contradiction.
\end{proof}
Let $E\to X$ be a smooth vector bundle
over a $\sigma$-compact, finite-dimensional
manifold~$X$ now, whose fibre is a finite-dimensional
$\K$-vector space~$F$.
The space of {\em distribution sections\/}
of $E$ is defined as the
strong dual $\cD'(X,E):=\cD^\infty(X,\Hom(E,\Omega_1(X)))'_b$,
where $\Omega_1(X)$ denotes the bundle
of real (resp., complex if $\K=\C$)
$1$-densities on~$X$. It is well-known that
$\cD'(X,E)$ is reflexive and thus barrelled.
We obtain a $C^\infty(X,\K)$-module structure
on $\cD'(X,E)$ via $f.u:=m_f'(u)\!: \sigma\mto u(f.\sigma)$.
In the same way, we
make $\cD'(X,E)$ a $\cD(X,\K)$-module.
Let us consider the simplest case first:
multiplication of smooth functions and
distributions on~$\R$.
As usual, given $\phi\in \cD(\R)$ and
$u\in \cD'(\R)$,
we write $\langle u,\phi\rangle:=u(\phi)$.
\begin{prop}\label{cheat}
The mapping
$\mu\!:C^\infty(\R)\times \cD'(\R)\to \cD'(\R)$,
$(f,u)\mto f.u:=u(f\sbull)$
and the corresponding map
$\nu\!: \cD(\R)\times \cD'(\R)\to \cD'(\R)$
are discontinuous.
\end{prop}
\begin{proof}
The topology on $\cD(\R)$ being finer than the
one induced by $C^\infty(\R)$,
it suffices to show that $\nu$ is discontinuous.
To see this, let us assume on the contrary that~$\nu$
is continuous, and derive a contradiction.
We choose $h\in \cD(\R)$ which is constantly~$1$
on some $0$-neighbourhood.
Then the polar $W:=\{h\}^0:=\{u\in \cD'(\R)\!: \; |\langle u,
h\rangle|\leq 1\}$
is a zero-neighbourhood in $\cD'(\R)$.
Thus, as we assume that $\nu$ is continuous,
there are zero-neighbourhoods
$V\sub \cD(\R)$ and $U\sub \cD'(\R)$
such that $\phi.u\in W$ for all
$\phi\in V$, $u\in U$.
Thus $|\langle V.U,h\rangle|=|\langle hU,V\rangle|\sub [0,1]$
and therefore
\begin{equation}\label{ppolar}
hU\sub V^0.
\end{equation}
Since $\cD(\R)$ is a reflexive
locally convex space and $V$ is a $0$-neighbourhood
in $\cD^r(\R)$ and thus absorbing,
we deduce from (\ref{ppolar})
that $hU$ is a weakly bounded subset of $\cD'(\R)$
and thus bounded by Mackey's Theorem (\cite{RFA}, Theorem~3.18).
Since, furthermore,
$\Supp(u)\sub \Supp(h)$ for all $u\in hU$,
the set $hU$ consists of distributions
whose orders are uniformly
bounded: there is $k\in \N_0$
such that every $u\in hU$
is continuous
on $\cD(\R)$, equipped with the topology
induced by $\cD^k(\R)$ (cf.\ \cite{Tre},
p.\,359, Theorem~34.3).
Now, as $U$ is a $0$-neighbourhood,
we find $0\not = r\in \R$
such that $r\delta_0^{(k+1)}\in U$,
where $\delta_0^{(k+1)}$ denotes the $(k+1)\,$st
derivative of the unit point-mass at the origin.
As $h\delta_0^{(k+1)}=\delta_0^{(k+1)}$,
we deduce that $r\delta_0^{(k+1)}\in hU$.
But this is a distribution of order $k+1$,
contradiction.
\end{proof}
In the general situation described before
Proposition~\ref{cheat}, we have:
\begin{prop}
The module multiplications
\[
\mu\!: C^\infty(X,\K)\times \cD'(X,E)\to \cD'(X,E)\]
and
\[
\nu\!:\cD^\infty(X,\K)\times \cD'(X,E)\to \cD'(X,E)\]
are hypocontinuous.
If $\dim(F)>0$ and $\dim(X)>0$,
then neither $\mu$ nor $\nu$ is continuous.
\end{prop}
\begin{proof}
It is not hard to see that
the bilinear mappings $\mu$ and $\nu$ are separately continuous.
As all of the locally convex spaces involved
are reflexive and thus barrelled,
the hypocontinuity follows from
\cite{Tre}, Theorem~41.2.

If $\dim(F)>0$ and $\dim(X)>0$,
then the discontinuity of
$\mu$ and $\nu$
can be shown by a
variant of the arguments used
to prove
Proposition~\ref{cheat}. We omit the details.
\end{proof}
In particular, we deduce
from Proposition~\ref{topALG},
Proposition~\ref{onlyifcomp}
and Proposition~\ref{cheat}:
\begin{cor}
The multiplication maps
\[
\begin{array}{ccccl}
C^\infty(\R) & \times & C^\infty(\R) & \to &
C^\infty(\R) \;\;\;\;\;\mbox{and}\\
\cD(\R) & \times & \cD(\R) & \to &
\;\,\cD(\R)
\end{array}
\]
are continuous.
In contrast, the
multiplication maps
\[
\begin{array}{ccccl}
C^\infty(\R) & \times & \cD(\R) & \to & \cD(\R),\\
C^\infty(\R) & \times & \cD'(\R) & \to & \cD'(\R)\;\;\;\;\;\;\;\;\;\mbox{and}\\
\cD(\R) & \times  & \cD'(\R) & \to & \cD'(\R)
\end{array}
\]
are hypocontinuous
and thus sequentially continuous,
but none of them is continuous.\Punkt
\end{cor}
See also \cite[p.\,13]{Mic} for the second of these
assertions, \cite[p.\,423]{Tre} for the third and fourth.\vspace{2 mm}\\
The algebra
$\cD(\R)$ of $\K$-valued
test functions is in fact a rather nice topological algebra:
its subset $Q(\cD(\R))$ of elements
possessing a quasi-inverse is an open $0$-neighbourhood,
and quasi-inversion $q\!: Q(\cD(\R))\to Q(\cD(\R))$ is a
$\K$-analytic mapping. This entails
that
the associated unital algebra
$\cD(\R)_e=\K \, e+ \cD(\R)$
has an open group of units,
with $\K$-analytic inversion
(and so $(\cD(\R)_e)^\times$ is a
$\K$-analytic Lie group).
We refer to~\cite{GOO} for more information
(cf.\ also \cite[p.\,65]{Mic}).
\section{Existence of universal central extensions}\label{sec8}
Throughout this section, $X$ denotes
a $\sigma$-compact, finite-dimensional smooth
manifold.
By abuse of notation,
given $f\in C^\infty(X):=C^\infty(X,\R)$,
we shall write $df$ for
the mapping $(x\mto df(x,\sbull))\in C^\infty(X,T^*X)$.
\begin{la}\label{lemmm}
The linear map
$\cD(X)\to \cD(X,T^*X)$, $f\mto df$
is continuous.
\end{la}
\begin{proof}
By the universal property
of direct limits
it is enough to show
that the linear mappings $C^\infty_K(X)\to
C^\infty_K(X,T^*X)$, $f\mto df$
are continuous for each compact subset~$K$ of~$X$.
This will hold if the corresponding map
$C^\infty(X)\to C^\infty(X,T^*X)$ is continuous.
Using Lemma~\ref{anyatlas},
we reduce to the case where $X$ is an open subset of~$\R^d$.
Now the desired continuity
is readily verified.
\end{proof}
\begin{prop}
The mapping $\cD(X)\times \cD(X)
\to \cD(X,T^*X)$,
$(f,g)\mto f \, dg$ is continuous.
\end{prop}
\begin{proof}
In fact, the mapping $\cD(X)\to \cD(X,T^*X)$,
$g\mto dg$ is continuous by Lemma~\ref{lemmm},
and $\cD(X,T^*X)$ is a topological
$\cD(X)$-module by Theorem~\ref{pushforw3}.
\end{proof}
The preceding proposition
entails
that the test function group
$\cD^\infty(X,G)$ (equipped with the
Lie group structure
defined in \cite{Glo}, cf.\ \cite{Alb},
which we recall in the next section)
has a universal central extension (in the sense of \cite{NCE})
in the category of smooth Lie groups modelled
on sequentially complete locally convex spaces,
for every $\sigma$-compact
finite-dimensional
smooth manifold~$X$ and semi-simple
finite-dimensional Lie group~$G$:
see \cite{CUR},
also~\cite{Mai}
for corresponding results on the Lie algebra level.
This intended application was one stimulus
for our investigations.
\section{Mapping groups on non-compact manifolds}\label{secmapgps}
If $G$ is a finite-dimensional
Lie group, it is well-known that
the group
$C^r(K,G)$ of $G$-valued $C^r$-functions
on a compact manifold~$K$ has a natural
analytic Lie group structure
modelled on the Lie algebra $C^r(K,L(G))$
(cf.\ \cite{Mil}).
For non-compact finite-dimensional manifolds~$X$,
one cannot expect to turn $C^r(X,G)$
into a Lie group modelled on $C^r(X,L(G))$.
However, even for infinite-dimensional~$G$,
there
is a natural Lie group structure
on the ``test function group'' $\cD^r(X,G)$
of those $G$-valued $C^r$-functions
which are constantly~$1$ off some compact set;
it is modelled on $\cD^r(X,L(G))$ (see~\cite{Glo},
cf.\ \cite{Alb}).
In the special situation when
$G$ is a $\K$-analytic Baker-Campbell-Hausdorff
Lie group, it was shown in~\cite{Glo}
that $C^r(X,G)$ can be made a $\K$-analytic Lie
group having $\cD^r(X,G)$ as an open subgroup
(and thus modelled on $\cD^r(X,L(G))$).
Using the more powerful technical tools developed in
the present paper, we now show that
$C^r(X,G)$ can be given a smooth (resp., $\K$-analytic)
Lie group structure modelled on $\cD^r(X,L(G))$,
for any smooth (resp, $\K$-analytic) Lie group~$G$
(which need not be BCH).
We remark that $C^\infty(X,G)$ had already
been given a Lie group structure
in the sense of convenient differential
calculus (again modelled on compactly supported functions)
when~$G$ is a Lie group in that sense
(cf.\ \cite{KaM}, Theorem~42.21).
This does not subsume our results,
as Lie groups in the broad sense of convenient differential
calculus need not be Lie groups in the more
conventional sense used in the present paper;
{\em e.g.\/},
they need not have continuous
group operations. Furthermore,
we can consider arbitrary $r\in \N_0\cup\{\infty\}$,
not only the special case $r=\infty$.
\begin{thm}\label{dstruckonCr}
Let $X$ be any $\sigma$-compact,
finite-dimensional $C^r$-manifold,
and $G$ be any smooth $($resp., $\K$-analytic$)$
Lie group. Then there is a unique smooth
$($resp., $\K$-analytic$)$ Lie group structure
on $C^r(X,G)$ such that
\[
\cD^r(X,V)\to C^r(X,G),\;\;\;
\gamma\mto \phi^{-1}\circ \gamma
\]
is a diffeomorphism of smooth $($resp, $\K$-analytic$)$
manifolds onto an open submanifold of $C^r(X,G)$,
for some chart $\phi\!: W\to V$ from an open
identity neighbourhood~$W$ in $G$
onto a balanced, open zero-neighbourhood
$V\sub L(G)$, with $\phi(1)=0$.
We have $L(C^r(X,G))\isom \cD^r(X,L(G))$.
\end{thm}
\begin{proof}
We recall from \cite{Glo} that there is
an (apparently unique)
smooth (resp., $\K$-analytic)
Lie group structure on $\cD^r(X,G)$ such that
$\Phi:=\cD^r(X,\phi^{-1})\!:
\cD^r(X,V)\to \cD^r(X,W)$,
$\gamma\mto \phi^{-1}\circ \gamma$
is a diffeomorphism of smooth $($resp, $\K$-analytic$)$
manifolds onto the open submanifold
$\cD^r(X,W)$ of $\cD^r(X,G)$,
for some chart $\phi\!: W\to V$
as described in the theorem.
Furthermore,
$L(\cD^r(X,G))\isom \cD^r(X,L(G))$.

Fix $\gamma\in C^r(X,G)$.
For every $x\in X$,
there is an open neighbourhood $P_x$
of $\gamma(x)$ in $G$ and
an open balanced $0$-neighbourhood $V_x\sub V$
in $L(G)$ such that $g \phi^{-1}(v) g^{-1}\sub W$
for all $g\in P_x$, $v\in V_x$.
We let $A_x\sub X$ be an open neighbourhood
of~$x$ such that $\gamma(A_x)\sub P_x$.
Then $U:=\bigcup_{x\in X}A_x\times V_x$
is an open $\R$-balanced neighbourhood
of $X\times \{0\}$ in $X\times L(G)$,
and $\wt{U}:=\bigcup_{x\in X}P_x\times V_x$
is an open subset of $G\times L(G)$
such that $(\gamma\times \id_{L(G)})(U)\sub \wt{U}$.
We define
\[
\wt{f}\!: \wt{U}\to L(G),\;\;\;\;
(g,v)\mto \phi(g\, \phi^{-1}(v)\,g^{-1})
\]
and let $f:= \wt{f}\circ (\gamma\times \id_{L(G)})|_U^{\wt{U}}:
U\to L(G)$. Then clearly $\wt{f}$ is a smooth
(resp., $\K$-analytic) mapping,
and $f$ is partially $C^\infty$ in the second argument,
with $d_2^jf$ of class $C^r$ for all $j\in \N_0$
(cf.\ proof of Proposition~\ref{proprealana}).
Therefore $f_*\!: \cD^r(X,U)\to \cD^r(X,L(G))$
is smooth (Theorem~\ref{pushforw3}).
In the case where~$G$ is a complex Lie group,
clearly
$f(x,\sbull)$ is complex analytic
and thus $d_2f(x,y,\sbull)$ complex linear,
whence $f_*$ is complex analytic
(Proposition~\ref{propcxana}).
The automorphism of groups
$J_\gamma\!:
\cD^r(X,G)\to \cD^r(X,G)$,
$\eta\mto \gamma\eta\gamma^{-1}$
satisfies
\[
\Phi^{-1}\circ J_\gamma|_S^{\cD^r(X,W)}
\circ \Phi|_{\cD^r(X,U)}^S=f_*\, ,\]
where $S:=\Phi(\cD^r(X,U))$.
Thus $J_\gamma$ is smooth (resp., complex analytic)
on the open identity
neighbourhhod~$S$
and therefore smooth (resp., complex
analytic when~$G$ is a complex Lie group).
Finally, if $G$ is a real analytic Lie group,
we deduce from Proposition~\ref{proprealana}
that $f_*|_{\cD^r(X,Q)}$ is real analytic for
some open, $\R$-balanced neighbourhood
$Q\sub U$ of the zero-section in $X\times L(G)$,
whence
$J_\gamma$ is real analytic on $\Phi(\cD^r(X,Q))$
and thus real analytic.

In either case,
the local characterization of
Lie groups (\cite{Glo}, Proposition~1.12)
shows that there is a unique smooth (resp.,
$\K$-analytic) Lie group structure
on $C^r(X,G)$ making $\cD^r(X,G)$ an open submanifold.
The remainder is then obvious.
\end{proof}
Note that,
by construction,
$\cD^r(X,G)$ is an open subgroup
of $C^r(X,G)$ and carries
the smooth (resp, $\K$-analytic) Lie group structure
described in~\cite{Glo}.\vspace{3 mm}\\
We may consider $C^r(X,\sbull)$ as a functor.
\begin{prop}\label{Cisfuncto}
Let $X$ be a $\sigma$-compact, finite-dimensional $C^r$-manifold,
$G_1$, $G_2$ be smooth $($resp., $\K$-analytic$)$
Lie groups, and $f\!: G_1\to G_2$ be a smooth $($resp.,
$\K$-analytic$)$ mapping such that $f(1)=1$.
Then
\[
C^r(X,f)\!: C^r(X,G_1)\to C^r(X,G_2),\;\;\;\;
\gamma \mto f\circ \gamma
\]
is a smooth $($resp., $\K$-analytic$)$ mapping
with respect to the manifold structures
on $C^r(X,G_1)$ and $C^r(X,G_2)$
described in Theorem~{\n \ref{dstruckonCr}\/}.
\end{prop}
\begin{proof}
For $j\in \{1,2\}$,
let $\phi_j\!: U_j\to V_j$ be a diffeomorphism
of smooth (resp., $\K$-analytic)
manifolds from an open zero-neighbourhood~$U_j$ in~$L(G_j)$
onto an open identity neighbourhood~$V_j$ in~$G_j$,
such that $\phi_j(0)=1$,
$\cD^r(X,V_j)$ is an open
identity neighbourhood in $\cD^r(X,G_j)$,
and $\Phi_j:=\cD^r(X,\phi_j)\!:
\cD^r(X,U_j)\to \cD^r(X,V_j)$
is a diffeomorphism of smooth (resp., $\K$-analytic)
manifolds. 
The mapping
\[
\tilde{g}\!: G_1\times U_1\to G_2,\;\;\;\;
\tilde{g}(a,u):=f(a)^{-1}f(a\cdot \phi_1(u))\]
is smooth (resp., $\K$-analytic),
and $\tilde{g}(a,0)=1$ for every $a\in G_1$. 
Suppose $\gamma\in C^r(X, G_1)$ is given.
Then, for every $x\in X$,
we find an open neighbourhood $P_x\sub G_1$ of $\gamma(x)$
and an open, balanced $0$-neighbourhood
$Q_x\sub U_1$ in $L(G_1)$
such that $\tilde{g}(P_x\times Q_x)\sub V_2$.
Let $B_x\sub X$ be an open neighbourhood
of~$x$ such that $\gamma(B_x)\sub P_x$.
Then $\wt{U}:=\bigcup_{x\in X} (P_x\times Q_x)\sub G_1\times U_1$
is an open neighbourhood of
$G_1\times \{0\}$ in $G_1\times L(G_1)$,
and $U:=\bigcup_{x\in X}(B_x\times Q_x)\sub X\times U_1$
is an open $\R$-balanced neighbourhood of
$X\times \{0\}$ in $X\times L(G_1)$
such that $(\gamma\times \id_{L(G_1)})(U)\sub \wt{U}$.
The mapping
$\tilde{h}:=\phi_2^{-1}\circ \tilde{g}|^{V_2}_{\wt{U}}$,
$\wt{U}\to U_2\sub L(G_2)$ is smooth (resp.,
$\K$-analytic).
We define
\begin{equation}\label{hpartsmooth}
h:= \tilde{h}\circ \, (\gamma\times \id_{L(G_1)})|_{U}^{\wt{U}},\;\;\;\;
U\to U_2.
\end{equation}
Thus 
$h(x,v):=\phi^{-1}_2[f(\gamma(x))^{-1}f(\gamma(x)\phi_1(v))]$
for all $(x,v)\in U$.\vspace{2 mm}\\
{\em Smooth or complex analytic case.\/}
The mappings $\tilde{h}$, $\gamma$, and $\id_{L(G_1)}$
being of class $C^\infty$, $C^r$, and $C^\infty$,
respectively, we easily deduce from Equation\,(\ref{hpartsmooth})
that~$h$ is partially $C^\infty$ in the second argument,
with $d_2^j$ of class~$C^r$,
for every $j\in \N_0$.
In the complex analytic case, furthermore
apparently $h(x,\sbull)$ is complex analytic
for each $x\in X$.
By Theorem~\ref{pushforw3} (resp., Proposition~\ref{propcxana}),
the mapping
$h_*\!: \cD^r(X, U)\to \cD^r(X, L(G_2))$ is
smooth (resp., complex analytic). We set $Q:=U$ in these cases.\vspace{2 mm}\\
In the {\em real analytic case\/},
$h_*|_{\cD^r(X,Q)}$ is real analytic for some
open $\R$-balanced
neighbourhood $Q\sub U$ of $X\times \{0\}$,
by Proposition~\ref{proprealana}.\vspace{2 mm}\\
In either case,
we let $\lambda_\gamma\!:C^r(X, G_1)\to C^{r}(X, G_1)$,
$\sigma\mto \gamma\cdot \sigma$  denote left translation
by $\gamma$ on $C^{r}(X, G_1)$
and $\lambda_{f\circ \gamma}$ denote left translation
by $f\circ \gamma$ on $C^{r}(X, G_2)$.
We abbreviate $W:=\Phi_1(\cD^r(X,Q))$.
Since
\[
\Phi^{-1}_2\circ
(\lambda_{f\circ \gamma}^{-1}\circ C^r(X, f)
\circ \lambda_\gamma)|_W^{\cD^r(X,V_2)}
\circ \Phi_1|_{\cD^r(X, Q)}^W
=h_*|_{\cD^r(X,Q)}
\]
is a smooth (resp., $\K$-analytic) mapping, it follows that
$\lambda_{f\circ \gamma}^{-1}\circ \, C^r(X, f)
\circ \, \lambda_\gamma$ is smooth (resp., $\K$-analytic)
on the identity neighbourhood $W$
of~$C^r(X, G_1)$.
Translations being diffeomorphisms,
this entails that $C^r(X, f)$
is smooth (resp., $\K$-analytic) on some neighbourhood of~$\gamma$.
\end{proof}
\begin{rem}
As a variant of the $C^0$-vector
bundles defined above,
we might allow $E$ to be a topological space
and the base~$X$ to be a $\sigma$-compact
locally compact space in Definition~\ref{defnbdle}
(instead of $C^0$-manifolds),
and replace the word ``$C^0$-diffeomorphism''
by ``homeomorphism.''
It is easy to see that
Theorem~\ref{pushforw2},
Theorem~\ref{pushforwK},
Theorem~\ref{pushforw3},
Proposition~\ref{propcxana}
and Proposition~\ref{proprealana}
remain valid for the variant of $C^0$-vector bundles
just described, by obvious adaptations
of the proofs.\footnote{Essentially,
the only difference in the proofs
is that we do not (and cannot)
use charts (like $\wt{\kappa}_n$ in {\bf \ref{defnUnVn}}).
Instead, we directly calculate
suprema over compact subsets of~$X$.
For example,
we define $\|\sigma\|_{n,q,0}$
in {\bf \ref{dstpomeg}}
directly as $\sup_{x\in \wb{V_n}}q(\sigma_{\psi_n}(x))$.
Since we are forced to use charts in the case of manifolds
discussed in the main body of the text,
but cannot use charts in the varied setting just
described,
we didn't find it practical to discuss both cases
simultaneosly.}
As a consequence,
we can give $C(X,G)=C^0(X,G)$ a smooth (resp.,
$\K$-analytic) Lie group structure
modelled on $\cD^0(X,L(G))$,
for every smooth (resp., $\K$-analytic)
Lie group~$G$ and $\sigma$-compact,
locally compact space~$X$ (which need not be a
topological manifold).
Proposition~\ref{Cisfuncto}
holds without changes for these groups. 
\end{rem}
{\em Acknowledgemants}.
This research was supported by Deutsche Forschungsgemeinschaft,
FOR 363/1-1.

{\footnotesize
{\bf Helge Gl\"{o}ckner}\\
Universit\"at Paderborn\\
Institut f\"ur Mathematik\\
Warburger Str.\ 100\\
33098 Paderborn, Germany\\
Email: glockner@math.upb.de}
\end{document}